\documentclass[11pt]{amsart}
\usepackage{amsmath}
\usepackage{amssymb}
\usepackage{amsfonts}
\usepackage{mathrsfs}
\usepackage{graphicx}
\usepackage{epsfig}
\usepackage[all]{xy}

\theoremstyle{plain}
\newtheorem{theorem}{Theorem}
\newtheorem{corollary}{Corollary}
\newtheorem{lemma}[theorem]{Lemma}
\newtheorem{proposition}[theorem]{Proposition}

\newtheorem{theoalph}{Theorem}

\theoremstyle{definition}

\newtheorem{remark}[theorem]{Remark}

\newtheorem{question}{Question}

\newtheoremstyle{citing}
  {3pt}
  {3pt}
  {\itshape}
  {}
  {\bfseries}
  {.}
  {.5em}
  {\thmnote{#3}}

\theoremstyle{citing}
\newtheorem*{generic}{}

\newcommand{\RR}{\mathbb{R}}

\newcommand{\CC}{\mathbb{C}}

\newcommand{\NN}{\mathbb{N}}

\newcommand{\ZZ}{\mathbb{Z}}

\def\M{{\mathcal M}}

\def\P{{\mathcal P}}

\newcommand{\parameter}{\lambda}
\newcommand{\expansion}[1]{\langle #1 \rangle}
\newcommand{\pcs}{post\nobreakdash-critical set}
\newcommand{\pcss}{post\nobreakdash-critical sets}
\newcommand{\sE}{\mathscr{E}}
\newcommand{\sP}{\mathscr{P}}
\newcommand{\Choquet}{\mathscr{C}}
\newcommand{\Sunimodal}{\textrm{S}\nobreakdash-unimodal}

\renewcommand{\=}{ : = }
\newcommand{\partn}[1]{{\smallskip \noindent \textbf{#1.}}}

\begin{document}

\title[Minimal post-critical sets]{Invariant measures of minimal post-critical sets of logistic maps}

\author[M.I. Cortez]{Mar{\'\i}a Isabel Cortez$^\dag$}
\author[J. Rivera-Letelier]{Juan Rivera-Letelier$^\ddag$}
\thanks{$\dag$Partially supported by Fondecyt de Iniciaci{\'o}n 11060002, Nucleus Millenius P04-069-F, and Research Network on Low Dimensional Dynamics, PBCT ACT-17, CONICYT, Chile.} 
\thanks{$\ddag$Partially supported by Research Network on Low Dimensional Dynamics, PBCT ACT-17, CONICYT, Chile.} 
\address{Mar{\'\i}a Isabel Cortez, Departamento de Matem{\'a}tica y Ciencia de la Computaci{\'o}n, Universidad de Santiago de Chile, Av. Libertador Bernardo O'Higgins~3363, Santiago, Chile.}
\email{mcortez@usach.cl}
\address{Juan Rivera-Letelier, Facultad de Matem{\'a}ticas, Campus San Joaqu{\'\i}n, P. Universidad Cat{\'o}lica de Chile, Avenida Vicu{\~n}a Mackenna~4860, Santiago, Chile}
\email{riveraletelier@mat.puc.cl}

\begin{abstract}
We construct logistic maps whose restriction to the $\omega$\nobreakdash-limit set of its critical point is a minimal Cantor system having a prescribed number of distinct ergodic and invariant probability measures.
In fact, we show that every metrizable Choquet simplex whose set of extreme points is compact and totally disconnected can be realized as the set of invariant probability measures of a minimal Cantor system corresponding to the restriction of a logistic map to the $\omega$\nobreakdash-limit set of its critical point.
Furthermore, we show that such a logistic map~$f$ can be taken so that each such invariant measure has zero Lyapunov exponent and is an equilibrium state of~$f$ for the potential~$-\ln |f'|$.
\end{abstract}

\maketitle{}
\section{Introduction}
The \textit{logistic family} of maps is defined for parameters $\parameter \in (0, 4]$, by
\begin{center}
\begin{tabular}{rcl}
$f_\parameter : [0, 1]$ & $\to$ & $[0, 1]$ \\
$x$ & $\mapsto$ & $\parameter x (1 - x).$
\end{tabular}
\end{center}
For each $\parameter \in (0, 4]$, the point $x = \tfrac{1}{2}$ is the unique point in $[0, 1]$ at which the derivative of~$f_{\parameter}$ vanishes.
We call $x = \tfrac{1}{2}$ the \textit{critical point} of~$f_\parameter$, and its $\omega$\nobreakdash-limit set is called the \textit{\pcs{}} of~$f_{\parameter}$.
It is a compact set that is forward invariant by~$f_{\parameter}$.

The following is our main result.
Recall that for a topological space~$X$, a continuous map $T : X \to X$ is said to be \textit{minimal}, if the forward orbit of each point in~$X$ is dense in~$X$.
\begin{generic}[Main Theorem]
Let~$\sE$ be a non-empty, compact, metrizable and totally disconnected topological space.
Then there is a parameter~$\parameter \in (0, 4]$ such that the \pcs{} of the logistic map~$f_{\parameter}$ is a Cantor set, the restriction of~$f_{\parameter}$ to this set is minimal, and such that the set of ergodic and invariant probability measures supported on this set, endowed with the weak$^*$ topology, is homeomorphic to~$\sE$.
\end{generic}
\begin{remark}
We have stated this result for the logistic family for simplicity.
We show that an analogous statement holds for each full family of unimodal maps, as well as for the family of symmetric tent maps.
See~\S\ref{ss:unimodal} and~\S\ref{ss:full families} for definitions, and Remark~\ref{r:tent family} for the proof in the case of the family of symmetric tent maps.
\end{remark}

The Main Theorem generalizes a result of H.~Bruin, that there is a parameter~$\parameter_0 \in (0, 4]$ such that the \pcs{} of~$f_{\parameter_0}$ is a Cantor set, and such that the restriction of~$f_{\parameter_0}$ to this set is minimal and possess at least two ergodic and invariant probability measures~\cite[Theorem~4]{Bru03}.
The Main Theorem shows that in fact~$\parameter_0$ can be taken so that the cardinality of the set of these measures is any prescribed finite number (taking~$\sE$ finite, of a given cardinality), countably infinite (taking for example $\sE = \{ \tfrac{1}{n} \mid n \ge 1 \} \cup \{ 0 \} \subset \RR$), or uncountable (taking for example~$\sE$ to be a Cantor set).\footnote{See~\S\ref{ss:notes and references} for an explicit description of the combinatorics of the logistic maps realizing these examples.}
The proof of the Main Theorem is based on the tools developed by Bruin in~\cite{Bru03}, and by Bruin, G.~Keller and M.~St.~Pierre in~\cite{BruKelsPi97}.

There are three sources of motivations for the Main Theorem.
The first is a result of T.~Downarowicz~\cite{Dow91}, that each metrizable Choquet simplex can be realized, up to an affine homeomorphism, as the space of invariant probability measures of a minimal map acting on a Cantor set;\footnote{A compact, convex and metrizable subset~$\Choquet$ of a locally convex real vector space is said to be a (metrizable) \textit{Choquet simplex}, if for each $v \in \Choquet$ there is a unique probability measure~$\mu$ that is supported on the set of extreme points of~$\Choquet$, and such that $\int x d\mu(x) = v$.
See for example~\cite[{\S}II.3]{Alf71} for several characterizations of Choquet simplex.
A well-known consequence of the ergodic decomposition theorem is that for each compact topological space~$X$ and each continuous map $T : X \to X$, the space of invariant probability measures of~$T$ is a metrizable Choquet simplex, see for example~\cite[p.~95]{Gla03}.} see also~\cite{GjeJoh00,Orm97}.
Thus the following question arises naturally.
\begin{question}\label{q:general simplex}
Is each metrizable Choquet simplex realizable, up to an affine homeomorphism, as the space of invariant probability measures of a minimal \pcs?
\end{question}
The Main Theorem gives a partial answer to this question.
In fact, it is well\nobreakdash-known that for each non-empty compact (metrizable) topological space~$\sE$ there is a unique (metrizable) Choquet simplex, up to an affine homeomorphism, whose set of extreme points is homeomorphic to~$\sE$.\footnote{In fact, if~$\mathscr{F}$ is a compact topological space then the space of probability measures supported on~$\mathscr{F}$ is a Choquet simplex whose set of extreme points is homeomorphic to~$\mathscr{F}$.
The uniqueness follows from a result of Bauer, that if~$\Choquet$ is a compact Choquet simplex whose set of extreme points is closed, then~$\Choquet$ is affine homeomorphic to the space of probability measures supported on the set of extreme points of~$\Choquet$. These results can be found for example in~\cite[Corollary~II.4.2]{Alf71}.}
So the Main Theorem gives a positive answer to Question~\ref{q:general simplex} in the case of metrizable Choquet simplices whose set of extreme points is compact and totally disconnected.

The second source of motivation comes from the ergodic theory of smooth maps.
Pesin theory is a powerful tool to study (ergodic) invariant measures with positive Lyapunov exponents.
However, in general there are few tools to study ergodic invariant measures with a zero Lyapunov exponent.
We say that an invariant measure of a logistic map is \textit{indifferent}, if it is ergodic and if its Lyapunov exponent is zero.
Using well-known results we show that the parameter~$\parameter \in (0, 4]$ in the Main Theorem can be taken such that in addition every ergodic invariant measure supported on the \pcs{} of~$f_\parameter$ is indifferent; see Appendix~\ref{a:ergodic theory}.
Thus the Main Theorem shows in particular that there is a logistic map having infinitely many distinct indifferent invariant probability measures, a result that to the best of our knowledge was unknown before.

The third source of motivation comes from the thermodynamic formalism.
Fix a parameter $\parameter \in (0, 4]$ and for a given invariant measure~$\mu$ denote by $h_\mu(f_\parameter)$ its measure theoretic entropy.
Recall that an invariant measure~$\mu$ is an \textit{equilibrium state} for the potential~$- \ln |f_{\parameter}'|$ if it realizes the supremum
$$ \sup \left\{ h_{\mu_0} - \int \ln |f'| d \mu_0 \right\}, $$
where $\mu_0$ runs through all the invariant probability measures of~$f_\parameter$.
Using well-known results we show that the parameter $\parameter \in (0, 4]$ in the Main Theorem can be taken such that in addition each invariant probability measure supported on the \pcs{} of~$f_\parameter$ is an equilibrium state for the potential $- \log |f_\parameter'|$; see Appendix~\ref{a:ergodic theory}.
Thus the Main Theorem provides examples of logistic maps having a large set of distinct ergodic equilibrium states, in sharp contrast with the (recent) related uniqueness results; see~\cite{BruKel98,BruTod07b,MakSmi03,PesSen08,PrzRiv08} and references therein.

We state these results in the following corollary for future reference.
The proof is a direct consequence of the proof of the Main Theorem and well\nobreakdash-known results; see Appendix~\ref{a:ergodic theory}.
\begin{corollary}\label{c:ergodic theory interval}
Let~$\sE$ be a non-empty, compact, metrizable and totally disconnected topological space.
Then there is~$\parameter \in (0, 4]$ verifying the conclusions of the Main Theorem and such that in addition the set of indifferent invariant probability measures of~$f_\parameter$ (resp. indifferent equilibrium states of~$f_\parameter$ for the potential $- \log |f_\parameter'|$) is homeomorphic to~$\sE$.
\end{corollary}
\begin{remark}\label{r:ergodic theory holomorphic}
We will now state for future reference a complex version of this result.
Its proof is similar to that of Corollary~\ref{c:ergodic theory interval}.
For a complex parameter~$\parameter \in \CC$ denote by~$P_\parameter$ the quadratic polynomial defined by
$$ P_\parameter (z) = \parameter z (1 - z), $$
viewed as a dynamical system acting on~$\CC$.
Let~$\sE$ be a non-empty, compact, metrizable and totally disconnected topological space.
Then there is a parameter $\parameter \in (0, 4]$ satisfying the conclusions of the Main Theorem, and such that in addition, if we denote by~$t_0 > 0$ the Hausdorff dimension of the Julia set of~$P_\parameter$, then the set of indifferent invariant probability measures (resp. indifferent equilibrium states for the potential $- t_0 \log |P_\parameter'|$) of~$P_\parameter$ is homeomorphic to~$\sE$.
\end{remark}
\subsection{Notes and references}\label{ss:notes and references}
Given a non-empty, compact, metrizable and totally disconnected topological space~$\sE$, we construct a rather explicit kneading map~$Q$ so that the conclusion of the Main Theorem holds for each unimodal map whose kneading map is equal to~$Q$, see~\S\ref{s:resonant kneading}.
(Unimodal maps are defined in~\S\ref{ss:unimodal} and kneading maps in~\S\ref{ss:cutting and kneading}.)
In fact we show that~$Q$ may be taken so that $Q(\NN) = \{ 2^{3^r} - 1 \mid r \ge 0 \}$, and so that there is a strictly increasing sequence of integers $(b(r))_{r \ge 0}$ such that $b(0) = 0$, $Q^{-1}(0) = \left\{ 0, \ldots, 2^{3^{b(1)}} - 2 \right\}$, and such that for each $r \ge 1$ we have
$$ Q^{-1}\left( 2^{3^r} - 2 \right)
=
\left\{ 2^{3^{b(r)}} - 1, \ldots, 2^{3^{b(r + 1)}} - 2 \right\}. $$
Furthermore, we have the following.
\begin{enumerate}
\item[1.]
If~$\sE$ is finite, then we can take the sequence $(b(r))_{r \ge 0}$ such that $b(0) = 0$, and such that for each $r \ge 1$ we have $b(r) = r - 1 + \#\sE$.
\item[2.]
If~$\sE = \{ \tfrac{1}{n} \mid n \ge 1 \} \cup \{ 0 \} \subset \RR$, then we can take\footnote{Here, for $x \in \RR$ we denote by $[x]$ the integer part of~$x$.}
$$ (b(r))_{r \ge 0} = \left(r + \left[ \frac{ \sqrt{8r + 1} - 1}{2} \right] \right)_{r \ge 0}. $$
\item[3.]
If~$\sE$ is a Cantor set, then we can take $(b(r))_{r \ge 0} = (2r)_{r \ge 0}$.
\end{enumerate}

All the unimodal maps we consider have a diverging and non-decreasing kneading map.
We suspect that for such maps the set of ergodic measures supported on the \pcs{} is compact with respect to the weak$^*$ topology.
So it is likely that to answer Question~\ref{q:general simplex} affirmatively in the general case one should have to consider unimodal maps having a kneading map that is not monotone.

See~\cite{GamMar06} for the realization of some concrete simplices as the space of invariant measures of minimal Cantor systems.

\subsection{Strategy and organization}
In this section we explain the strategy of the proof of the Main Theorem and simultaneously describe the organization of the paper.

We use the fact that the \pcs{} of a unimodal map whose kneading map diverges is a Cantor set where the unimodal map is minimal (Proposition~\ref{p:continuous unimodal}).
We recall the definition of ``kneading map'', as well as other concepts and results about unimodal maps, in~\S\ref{s:preliminaries}.

In~\S\ref{s:resonant kneading} we introduce a class of diverging kneading maps we call ``resonant'', and then prove the Main Theorem assuming a result describing, for a unimodal map whose kneading map is resonant and satisfies an additional property, the space of invariant probability measures supported on its (minimal) \pcs{} (Theorem~\ref{t:invariant of resonant}).

To prove Theorem~\ref{t:invariant of resonant} we first recall in~\S\ref{ss:kneading odometer} the generalized odometer associated to a kneading map, that was introduced in~\cite{BruKelsPi97}.
We show that for a unimodal map whose kneading map~$Q$ is non-decreasing and diverging the space of invariant probability measures supported on the \pcs{} is affine homeomorphic to the space of invariant probability measures of the generalized odometer associated to~$Q$ (Theorem~\ref{t:measure isomorphism} in \S\ref{ss:measure isomorphism}).

In~\S\ref{s:Bratteli-Vershik system} we recall the definition of Bratteli-Vershik system associated to a kneading map, that was introduced by Bruin in~\cite{Bru03}.
We recall in particular the representation of the space of invariant probability measures of such a system, as the inverse limit of some explicit linear maps, called ``transition matrices'' (\S\ref{ss:invariant measures}).

The proof of Theorem~\ref{t:invariant of resonant} is given in~\S\ref{ss:proof of invariant of resonant}, after calculating the transition matrices of the corresponding Bratteli-Vershik system in~\S\ref{ss:transition matrices}.

In Appendix~\ref{a:ergodic theory} we give the proof of Corollary~\ref{c:ergodic theory interval}, using well\nobreakdash-known results in the literature.

\subsection{Acknowledgments}
Corollary~\ref{c:ergodic theory interval} and the result described in Remark~\ref{r:ergodic theory holomorphic} give a partial answer to a question posed by Feliks Przytycki in several discussions with JRL over the years.
We are also grateful with Henk Bruin and Neil Dobbs for useful comments concerning a first version of this paper, with Jan Kiwi for evoking Proposition~\ref{p:continuous unimodal}, and with Tomasz Downarowicz, Eli Glasner, Godofredo Iommi, Jan Kiwi and Bartlomiej Skorulski for useful discussions.
Finally, we thank the referee for useful comments.

MIC is grateful with the Department of Mathematics of the University of Washington, and JRL with the Institute of Mathematics of the Polish Academy of Sciences (IMPAN), where part of this work was done.

\section{Preliminaries}\label{s:preliminaries}
In this section we fix some notation (\S\ref{ss:linear algebra notation}), and then recall some definitions and results about unimodal maps (\S\S\ref{ss:unimodal}, \ref{ss:cutting and kneading}, \ref{ss:full families}).
See~\cite{BruBru04,dMevSt93} for background on unimodal maps.

Throughout this rest of the paper we denote by~$\NN$ the set of non\nobreakdash-negative integers.
\subsection{Linear algebra notation}\label{ss:linear algebra notation}
Given a non-empty finite set~$V$, for each $v \in V$ we denote by $\vec{e}_v \in \RR^V$ the vector having all of its coordinates equal to~$0$, except for the coordinate corresponding to~$v$ that is equal to~$1$.
Notice in particular that $\{ \vec{e}_v \mid v \in V \}$ is a base of~$\RR^V$.
Furthermore we will denote by~$\Delta_{V}$ the \textit{unit simplex} in $\RR^V$, which is defined as the (closed) convex hull of $\{ \vec{e}_v \mid v \in V \}$ in~$\RR^V$, and by~$\| \cdot \|_1$ the norm on~$\RR^V$ defined by $\| \sum_{v \in V}\alpha_v \vec{e}_v \|_1 = \sum_{v \in V} |\alpha_v|$.
Observe that~$\| \cdot \|_1$ is constant equal to~1 on~$\Delta_V$. 

Given non-empty finite sets $V, V'$ denote by $\M_{V, V'}$ the group of matrices whose entries are real and indexed by $V \times V'$.
For a matrix $A \in \M_{V, V'}$ we denote by $A^t$ the transpose of~$A$, and for $(v, v') \in V \times V'$ we denote by $A(v, v')$ the corresponding entry of~$A$, and by $A(\cdot, v')$ the corresponding column vector of~$A$.
Given column vectors $\{ \vec{x}_{v'} \mid v' \in V' \}$ in $\RR^V$ we denote by $(\vec{x}_{v'})_{v' \in V'}$ the matrix in $\M_{V, V'}$ whose column vector corresponding to the coordinate~$v'$ is equal to $\vec{x}_{v'}$.

We say that a matrix~$A$ is (left) \textit{stochastic} if all of its entries are non-negative and if the sum of all the entries in each column is equal to~1.
Observe that a stochastic matrix in $\M_{V, V'}$ maps~$\Delta_{V'}$ into~$\Delta_V$, and that the product of stochastic matrices is stochastic.
\begin{lemma}\label{l:non expanding}
Let $V, V'$ be non-empty finite sets and let $A \in \M(V, V')$ be a stochastic matrix.
Then for each $\vec{w}, \vec{w}' \in \Delta_V$ we have
$$ \| A(\vec{w}) - A(\vec{w}') \|_1 \le \| \vec{w} - \vec{w}' \|_1. $$
\end{lemma}
\begin{proof}
Putting $\vec{w} = ( w_v )_{v \in V}$ and $\vec{w}' = ( w_{v}' )_{v \in V'}$, we have
\begin{multline*}
\| A(\vec{w}) - A(\vec{w}') \|_1
\le
\sum_{v \in V} \| (w_v - w_v') A(\vec{e}_v) \|_1
=
\sum_{v \in V} |w_v - w_v'|
=
\| \vec{w} - \vec{w}' \|_1.
\end{multline*}
\end{proof}
\subsection{Unimodal maps}\label{ss:unimodal}
A continuous map $f:[0,1]\to [0,1]$ is \textit{unimodal} if~$f(0) = f(1) = 0$, and if there exists a point $c \in [0,1]$ such that~$f$ is strictly increasing on~$[0, c]$, and strictly decreasing on~$[c, 1]$.
The point~$c$ is called the \textit{turning} or \textit{critical point} of~$f$.

For each $\parameter \in (0, 4]$ the logistic map~$f_\parameter$ is a unimodal with critical point~$x = \tfrac{1}{2}$.
Other well\nobreakdash-known examples of unimodal maps are the \textit{symmetric tent maps}, which are defined for a parameter $s \in (0, 2]$,  by
$$ T_s(x) = \begin{cases}
s x & \text{if $x \in [0, \tfrac{1}{2}]$}, \\
s(1 - x) & \text{if $x \in [\tfrac{1}{2}, 1]$}.
\end{cases} $$

Let~$f$ be a unimodal map with critical point~$c$.
The $\omega$\nobreakdash-limit of~$c$ will be called the \textit{\pcs{}} of~$f$.
When either $f(c) \le c$ or $f^2(c) \ge c$, it is easy to see that the \pcs{} of~$f$ reduces to a single point.
We will thus (implicitly) assume from now on that for each unimodal map~$f$ that we consider we have $f^2(c) < c < f(c)$.

\subsection{Cutting times and the kneading map}\label{ss:cutting and kneading}
To describe the dynamics of a unimodal map~$f$ on its \pcs, we will make the following definitions.
Let~$c$ be the critical point of~$f$ and for each $n \ge 1$ put $c_n = f^n(c)$.
Define the sequence of compact intervals $(D_n)_{n \ge 1}$ inductively by $D_1 = [c, c_1]$, and for each $n \ge 2$, by
$$
D_{n} =
\begin{cases}
f(D_{n - 1}) & \text{if } c \not \in D_{n - 1}, \\
[c_{n}, c_1] & \text{otherwise}.
\end{cases}
$$
An integer $n \ge 1$ will be called a \textit{cutting time} if $c \in D_n$.
We will denote by $(S_k)_{k \ge 0}$ the sequence of all cutting times.
From our assumption that $f^2(c) < c < f(c)$ it follows that $S_0 = 1$ and $S_1 = 2$.

It can be shown that if~$S$ and~$S'> S$ are consecutive cutting times, then $S' - S$ is again a cutting time, and that this cutting time is less than or equal to~$S$ when~$f$ has no periodic attractors, see for example~\cite{Bru95,Hof80}.
That is, if~$f$ has no periodic attractors then for each $k \ge 1$ there is a non-negative integer~$Q(k)$, such that $Q(k) \le k - 1$, and
$$ S_k - S_{k - 1} = S_{Q(k)}. $$
Putting $Q(0) = 0$, the function $Q : \NN \to \NN$ so defined is called the \textit{kneading map} of~$f$.
It follows from the recursion formula above, and from $S_0 = 1$, that the sequence $(S_k)_{k \ge 0}$ of cutting times is determined by~$Q$.

We will say that a function $Q : \NN \to \NN$ is a \textit{kneading map} if there is a unimodal map~$f$ with critical point~$c$, such that $f^2(c) < c < f(c)$, such that~$f$ has no periodic attractors and such that the kneading map of~$f$ is equal to~$Q$.
If we denote by~~$\succeq$ the lexicographical ordering in~$\NN^\NN$, then a function $Q : \NN \to \NN$ is a kneading map if and only if $Q(0) = 0$, for each $k \ge 1$ we have $Q(k) \le k - 1$, and if for each $k \ge 1$ we have
\begin{equation}
\label{admisible}
\{Q(k+j)\}_{j\geq 1}
\succeq
\{Q(Q (Q(k)) + j)\}_{j\geq 1},
\end{equation}
see~\cite{Bru95,Hof80}.
Notice in particular that, if $Q : \NN \to \NN$ is non-decreasing, $Q(0) = 0$ and for each $k \ge 1$ we have $Q(k) \le k - 1$, then~$Q$ is a kneading map.
\subsection{Full families of unimodal maps}\label{ss:full families}
We will say that a family of unimodal maps $(g_t)_{t \in I}$ is \textit{full}, if for each kneading map~$Q$ there is a parameter $t \in I$ such that the kneading map of~$g_t$ is equal to~$Q$.\footnote{Full families are usually defined with ``kneading sequences'', as opposed kneading maps.
The definition adopted here gives a weaker condition, but it is enough for our purposes.}
The logistic family $(f_\parameter)_{\parameter \in (0, 4]}$ is well\nobreakdash-known to be full.
In~\cite[Theorem~4]{HofKel90} it is shown that a family of~$C^1$ unimodal maps $(g_t)_{t \in [0, 1]}$ satisfying the following properties is full: both~$g_t(x)$ and~$g_t'(x)$ vary continuously when $(t, x)$ varies on $[0, 1] \times [0, 1]$, for each $t \in [0, 1]$ (resp. $t \in (0, 1]$) the critical point~$c_t$ of~$g_t$ satisfies $g_t(c_t) > c_t$ (resp. $g_t^2(c_t) < c_t$), and we have~$g_0^2(c_0) = c_0$ and~$g_1(c_1) = 1$; see also~\cite[{\S}III.1]{ColEck80},~\cite[{\S}II.4]{dMevSt93}.

In the following proposition we gather several known results.
\begin{proposition}\label{p:continuous unimodal}
Let~$f$ be a unimodal map whose kneading map diverges.
Then the \pcs{} of~$f$ is a Cantor set, and the restriction of~$f$ to this set is minimal and has zero topological entropy.
Furthermore, if~$\widehat{f}$ is a unimodal map having the same kneading map as~$f$, then the space of invariant probability measures of~$\widehat{f}$ supported on the \pcs{} of~$\widehat{f}$ is affine homeomorphic to that of~$f$.
\end{proposition}
\begin{proof}
As the logistic family is full there is a parameter $\parameter \in (0, 4]$ such that the kneading map of~$f_\parameter$ is equal to that of~$f$.
The first part of the lemma is shown for~$f_\parameter$, for example in~\cite[Proposition~3.1]{Bru98} and~\cite[Proposition~1]{Bru03},~\cite[\S11]{BloLyu91}.
To show it holds for a general unimodal map~$f$, denote by $X_f$ and~$X_{f_\parameter}$ the \pcs{} of~$f$ and~$f_\parameter$, respectively.

Consider a non-decreasing and continuous map $h : [0, 1] \to [0, 1]$ mapping the critical point of~$f$ to that of~$f_\parameter$, and such that $h \circ f = f_\parameter \circ h$.
Such a map is given by Milnor-Thurston's theory, see for example~\cite[\S{}III.4]{dMevSt93}.
For each $x \in [0, 1]$ the set $h^{-1}(x)$ is either reduced to a point, or it is a closed interval.
Denote by~$X_0$ the subset of~$X_{f_\parameter}$ of those~$x$ such that~$h^{-1}(x)$ is reduced to a point.
Observe that the set $X_{f_\parameter} \setminus X_0$ is countable and that~$h$ is injective on~$h^{-1}(X_0)$.
As~$f_\parameter$ is minimal on~$X_{f_\parameter}$ and this set is a Cantor set, $f_\parameter$ has no periodic points on~$X_{f_\parameter}$.
Therefore for each~$x \in X_{f_\parameter} \setminus X_0$ the intervals $(f^n(h^{-1}(x)))_{n \ge 0}$ are pairwise disjoint.
In particular the length of $f^n(h^{-1}(x))$ converges to~$0$ as $n \to + \infty$.
It follows that~$X_f$ is equal to boundary of~$h^{-1}(X_{f_\parameter})$ and~$f$ is minimal on~$X_f$.
Since~$X_{f_\parameter}$ is a Cantor set, the set~$X_f$ is also a Cantor set.
To prove that the topological entropy of~$f|_{X_f}$ is zero, by the variational principle we just need to show that the measure theoretic entropy of each invariant measure of~$f$ that is supported on~$X_f$ is equal to zero.
To do this we use again the fact that for each $x \in X_{f_\parameter} \setminus X_0$ the intervals $(f^n(h^{-1}(x)))_{n \ge 0}$ are pairwise disjoint.
This implies that~$h^{-1}(x)$ has measure zero for each invariant measure of~$f$.
Thus each invariant measure of~$f$ supported on~$h^{-1}(X_f)$ is in fact supported on~$h^{-1}(X_0)$.
So the measure theoretic entropy of an invariant measure~$\mu$ of~$f$ that is supported on~$X_f$, is equal to the measure theoretic entropy of the invariant measure~$h_*(\mu)$ of~$f_\parameter$.
As the topological entropy of~$f_\parameter|_{X_{f_\parameter}}$ is equal to zero, the variational principle implies that the measure theoretic entropy of~$h_*(\mu)$, and hence that of~$\mu$, are both zero.

To prove the last statement of the lemma, it is enough to consider the case $\widehat{f} = f_\parameter$.
The assertion follows from the fact that each invariant measure of~$f$ that is supported on~$X_f$, is in fact supported on~$h^{-1}(X_0)$.
\end{proof}

\section{Resonant kneading maps}\label{s:resonant kneading}
In this section we introduce a class of diverging kneading maps we call ``resonant'', and then reduce the proof of the Main Theorem to a result describing, for a unimodal map whose kneading map is resonant and satisfies an additional property, the space of invariant probability measures supported on the \pcs{}.
This result is stated as Theorem~\ref{t:invariant of resonant}, below, and its proof occupies \S\S\ref{s:measure isomorphism}, \ref{s:Bratteli-Vershik system}, \ref{s:invariant of resonant}.

We will say that a function $Q : \NN \to \NN$ is \textit{resonant}, if it is non-decreasing, diverging, and if every integer $k \in \NN$ satisfying $Q(k + 1) > Q(k)$ belongs to the image of~$Q$.
Observe that if~$Q$ is a resonant function such that for each~$k \ge 1$ we have $Q(k) \le k - 1$, then~$Q$ is a kneading map, see~\S\ref{ss:cutting and kneading}.

For each resonant kneading map~$Q$ satisfying $Q(0) = 0$ there are strictly increasing sequences $(q_r)_{r \ge 0}$ and $(b(r))_{r \ge 0}$ of integers such that $q_0 = b(0) = 0$, $Q(\NN) = \{ q_r \mid r \ge 0 \}$, $Q^{-1}(0) = \{ 0, \ldots, q_{b(1)} \}$, and such that for every $r \ge 1$ we have
$$ Q^{-1}(q_r) = \{ q_{b(r)} + 1, \ldots, q_{b(r + 1)} \}. $$
Conversely, each pair of strictly increasing sequences of integers $(q_r)_{r \ge 0}$ and $(b(r))_{r \ge 0}$ satisfying~$q_0 = b(0) = 0$, define in this way a resonant kneading map~$Q: \NN \to \NN$.

\begin{theoalph}\label{t:invariant of resonant}
Let $(q_r)_{r \geq 0}$ and $(b(r))_{r \geq 0}$ be the strictly increasing sequences of positive integers such that $q_0 = b(0) = 0$, and let~$Q$ be the resonant kneading map defined by~$Q^{-1}(0) = \{ 0, \ldots, q_{b(1)} \}$, and for $r \ge 1$,  by
$$ Q^{-1}(q_r) = \{ q_{b(r)} + 1, \ldots, q_{b(r + 1)} \}. $$
Define $(S_k)_{k \ge 0}$ recursively by $S_0 = 1$ and $S_k = S_{k - 1} + S_{Q(k)}$, and assume that,
\begin{equation}
\label{e:positive determinant}
\prod_{r \geq 0} \left(1 - \frac{S_{q_r}}{S_{q_{r + 1}}} \right)>0.
\end{equation}
Furthermore, for each $r \ge 0$ put $I_r = \{ 0, \ldots, b(r) - r \}$ and let $\Xi_r : \RR^{I_{r + 1}} \to \RR^{I_r}$ be the stochastic linear map defined by
$$
\Xi_{r}(x_0,\ldots, x_{b(r + 1)- (r + 1)})
=
\left( \left(\sum_{j=b(r) - r}^{b(r + 1)- (r + 1)}x_j \right), x_0,\ldots,x_{b(r) - r - 1} \right).
$$
Then for each unimodal map~$f$ whose kneading map is equal to~$Q$, the \pcs{} of~$f$ is a Cantor set, $f$ is minimal on this set, and the space of invariant probability measures of~$f$ supported on this set is affine homeomorphic to $\varprojlim_{r} (\Delta_{I_r}, \Xi_r)$.
\end{theoalph}
We will now prove the Main Theorem assuming Theorem~\ref{t:invariant of resonant}.
We will use the following lemma, whose proof is at the end of this section.
Note that~$(I_r)_{r \ge 0}$ and~$(\Xi_r)_{r \ge 0}$, defined in the statement of Theorem~\ref{t:invariant of resonant}, only depend on $(b(r))_{r \ge 0}$.
\begin{lemma}\label{l:b}
Let~$\sE$ be a non-empty, compact, metrizable and totally disconnected topological space.
Then there is a strictly increasing sequence of integers $(b(r))_{r \ge 0}$ such that $b(0) = 0$, and such that the set of extreme points of $\varprojlim_{r} (\Delta_{I_r}, \Xi_r)$ is homeomorphic to~$\sE$.
\end{lemma}
Let~$\sE$ be a non-empty, compact, metrizable and totally disconnected topological space and let $(b(r))_{r \ge 0}$ be given by Lemma~\ref{l:b}.
A direct computation using the Lemma~\ref{l:q}, below, shows that the sequences $(q_r)_{r \ge 0} \= \left( 2^{3^{r}} - 2 \right)_{r \ge 0}$ and~$(b(r))_{r \ge 0}$ satisfy~\eqref{e:positive determinant}, so the Main Theorem follows from Theorem~\ref{t:invariant of resonant} and the fact that the logistic family is full.
\begin{remark}\label{r:tent family}
We will now show that the Main Theorem holds when one replaces the logistic family by the family of symmetric tent maps $(T_s)_{s \in (0, 2]}$, defined in~\S\ref{ss:unimodal}.
When~$\sE$ is reduced to a single point consider the ``Fibonacci parameter'' $\parameter \in (0, 4]$, characterized by the property that the kneading map~$Q$ of the corresponding logistic map is given by $Q(k) = \max \{ 0, k - 2 \}$.
It is well\nobreakdash-known that~$f_\parameter$ is not renormalizable, that the \pcs{} of~$f_\parameter$ is a Cantor set, and that the restriction of~$f_\parameter$ to this set is minimal and uniquely ergodic, see for example~\cite[Corollary~1]{Bru03}.
It is also well\nobreakdash-known that there is a parameter~$s \in (0, 2]$ such that $f_\parameter$ is topologically conjugated to the tent map~$T_s$, see for example~\cite[{\S}III.4]{dMevSt93}.
So, when~$\sE$ is reduced to a single point, the parameter~$s$ satisfies the desired properties.

Suppose now that~$\sE$ contains at least~2 points, and let~$(b(r))_{r \ge 0}$ be the sequence given by Lemma~\ref{l:b}.
As~$\sE$ contains at least 2 points there is~$r_0 \ge 1$ such that for all $r \ge r_0$ we have $b(r) \ge r + 1$.
Modifying the values of~$b$ for $r = 1, \ldots, r_0 - 1$, if necessary, we assume that~$r_0 = 1$.
Let~$\parameter \in (0, 4]$ be a parameter such that the kneading map~$Q$ of the logistic map~$f_\parameter$ is the resonant function defined by~$(q_r)_{r \ge 0} \= (2^{3^r} - 2)_{r \ge 0}$ and~$(b(r))_{r \ge 0}$.
By the argument given above this remark, the parameter~$\parameter$ satisfies the conclusions of the Main Theorem.
As for each $r \ge 1$ we have $b(r) \ge r + 1$, it is easy to see that for each $k \ge 2$ we have $Q(k) \le k - 2$.
It thus follows that $f_\parameter$ is not renormalizable~\cite[Lemma~2.3]{Bru98} and that there is a parameter~$s \in (0, 2]$ such that the tent map~$T_s$ is topologically conjugated to~$f_\parameter$.
Hence, in this case the parameter~$s$ satisfies the desired properties.
\end{remark}
\begin{lemma}\label{l:q}
Let $(q_r)_{r \ge 0}$ be a strictly increasing sequence of integers such that $q_0 =0$ and such that for each sufficiently large $r \ge 1$ we have
\begin{equation}\label{e:norma}
q_{r + 1}
\ge
q_{r} + r^2 \prod_{s = 0}^{r - 1} (1 + q_{s + 1} - q_{s}).
\end{equation}
Given a strictly increasing sequence of integers $(b(r))_{r \ge 0}$ such that $b(0) = 0$, let~$Q$ be the resonant function defined by $(q_r)_{r \ge 0}$ and $(b(r))_{r \ge 0}$.
If we define the sequence $(S_k)_{k \ge 0}$ recursively by $S_0 = 1$ and $S_k = S_{k - 1} + S_{Q(k)}$, then we have
$$
\prod_{r \geq 0}
\left(1-\frac{S_{q_r}}{S_{q_{r + 1}}}\right) > 0.
$$
\end{lemma}
\begin{proof}
For each $r \ge 0$ and $k \in \{ q_{r} + 1, \ldots, q_{r + 1} \}$ we
have $Q(k)=Q(q_{r + 1})$.
Using the recursion formula $S_k = S_{k - 1} + S_{Q(k)}$ inductively, we obtain
\begin{equation}
\label{obs0}
S_{q_{r + 1}}=S_{q_{r}}+(q_{r + 1} - q_{r})S_{Q(q_{r + 1})},
\end{equation}
so
\begin{equation*}
\frac{S_{q_{r}}}{S_{q_{r + 1}}}
=
\left(1 + (q_{r + 1} - q_{r})\frac{S_{Q(q_{r + 1})}}{S_{q_r}} \right)^{-1}.
\end{equation*}
It is thus enough to show that for each~$r \ge 0$ for which~\eqref{e:norma} is satisfied we have
$(q_{r + 1} - q_r)\frac{S_{Q(q_{r + 1})}}{S_{q_r}} \ge r^2$, since this implies that $\sum_{r \ge 1} \frac{S_{q_r}}{S_{q_{r + 1}}} < + \infty$ and therefore that $\prod_{r \geq 0} \left(1-\frac{S_{q_r}}{S_{q_{r + 1}}}\right) > 0.$

From \eqref{obs0}, with~$r$ replaced by~$r - 1$, and from the inequality $Q(q_r)=Q(q_{r - 1} + 1)\leq q_{r - 1}$, we obtain
$$
S_{q_r}\leq
S_{q_{r - 1}}(1+ q_r - q_{r - 1}).
$$
So by induction in~$r$ we get,
\begin{equation*}
S_{q_r}\leq \prod_{s=0}^{r-1}(1+q_{s + 1} - q_s).
\end{equation*}
Hence by~\eqref{e:norma} we have $q_{r + 1} - q_r \geq r^2 S_{q_r} \geq r^2 \frac{S_{q_r}}{S_{Q(q_{r + 1})}},$ which gives the desired inequality.
\end{proof}
\begin{proof}[Proof of Lemma~\ref{l:b}]
Let $( \sP_j )_{j \ge 1}$ be a sequence of partitions of~$\mathscr{E}$ into clopen sets that generate the topology of~$\mathscr{E}$, in such a way that for each $j \ge 1$, the partition~$\sP_{j + 1}$ is finer than~$\sP_j$.
For each $j \ge 1$ and $P \in \sP_{j + 1}$ denote by~$\ell_j(P)$ the element of~$\sP_j$ containing~$P$.
Note that the map $\ell_j : \sP_{j + 1} \to \sP_j$ so defined is onto, and that the inverse limit $\varprojlim_{j} (\sP_j, \ell_j)$ is homeomorphic to~$\sE$.
Denote by $L_j : \RR^{\sP_{j + 1}} \to \RR^{\sP_j}$ the stochastic linear map such that for each $P \in \sP_{j + 1}$ we have $L_j (\vec{e}_P) = \vec{e}_{\ell_j(P)}$.
It is straightforward to check that the set of extreme points of $\varprojlim_{j} (\Delta_{\sP_j}, L_j)$ is equal to
$$ \varprojlim_{j} \left( \{ \vec{e}_P \mid P \in \sP_j \}, L_j|_{ \{ \vec{e}_P \mid P \in \sP_{j + 1} \} } \right),$$
which is clearly homeomorphic to $\varprojlim_{j} (\sP_j, \ell_j)$, and in turn to~$\sE$.
So we just need to find a sequence $(b(r))_{r \ge 0}$ in such a way that $\varprojlim_{j} (\Delta_{I_r}, \Xi_r)$ is affine homeomorphic to~$\varprojlim_{j} (\Delta_{\sP_j}, L_j)$.

Put $r(1) = 1$ and for $j \ge 2$ put $r(j) = 1 + \sum_{i = 1}^{j - 1} \# \sP_i$.
Then we put $b(0) = 0$, for $j \ge 1$ we put $b(r(j)) = r(j + 1) - 1$ and, as in the statement of Theorem~\ref{t:invariant of resonant}, we put
$$ I_{r(j)} = \{0, \ldots, b(r(j)) - r(j) \}
=
\{ 0, \ldots, \# \sP_j - 1 \}. $$
To define $b(r)$ for $r \neq r(j)$, we will define for each $j \ge 1$ a bijection $\gamma_j : \sP_j \to I_{r(j)}$ inductively as follows.
Let~$\gamma_1$ be an arbitrary bijection between~$\sP_1$ and $I_{r(1)}$.
Let $j \ge 1$ be given and assume by that the bijection~$\gamma_j$ is already defined.
Then we let $\gamma_{j + 1} : \sP_{j + 1} \to I_{r(j + 1)}$ be any bijection such that for each $i \in I_{r(j)}$ it maps $(\ell_j \circ \gamma_j)^{-1}(i)$ onto
$$ \{k \in \NN \mid \#(\ell_j \circ \gamma_j)^{-1}( \{0, \ldots, i - 1 \}) \le k \le \#(\ell_j \circ \gamma_j)^{-1}( \{0, \ldots, i \}) - 1 \}. $$
We complete the definition of the sequence $(b(r))_{r \ge 0}$, by putting for each $j \ge 1$ and $i \in \{ 1, \ldots, \#\sP_j - 1 \}$,
$$ b(r(j) + i) = b(r(j)) + \#(\ell_j \circ \gamma_j)^{-1}( \{0, \ldots, i - 1 \}). $$

To prove that the sequence $(b(r))_{r \ge 0}$ satisfies the desired property, define for each $r \ge 0$ the map $\xi_r : I_{r + 1} \to I_r$, by
$$ \xi_r^{-1} (0) = \{ b(r) - r, b(r) - r + 1, \ldots, b(r + 1) - (r + 1) \}, $$
and for each $i \in I_r \setminus \{ 0 \}$ by $\xi_r^{-1}(i) = i - 1$.
Note that for each $i \in I_{r + 1}$ we have $\Xi_r(\vec{e}_i) = \vec{e}_{\xi_r(i)}$.
Then it is easy to check that for each $j \ge 1$ we have
$$ \ell_j \circ \gamma_j = (\xi_{r(j)} \circ \cdots \circ \xi_{r(j + 1) - 1}) \circ \gamma_{j + 1}.$$
So, if for each $j \ge 1$ we denote by $\Gamma_j : \RR^{\sP_j} \to \RR^{I_{r(j)}}$ the stochastic linear map such that for each $P \in \sP_j$ we have $\Gamma_j (\vec{e}_P) = \vec{e}_{\gamma_j(P)}$, then
$$ L_j \circ \Gamma_j = (\Xi_{r(j)} \circ \cdots \circ \Xi_{r(j + 1) - 1}) \circ \Gamma_{j + 1}.$$
It follows that the sequence of linear maps $(\Gamma_j)_{j \ge 1}$ induces an affine homeomorphism between $\varprojlim_{j} (\Delta_{\sP_j}, L_j)$ and $\varprojlim_{r} (\Delta_{I_r}, \Xi_r)$.
\end{proof}

\section{Generalized odometers and \pcss}\label{s:measure isomorphism}
The purpose of this section is to prove that for a unimodal map whose kneading map~$Q$ is non-decreasing and diverging, the space of invariant probability measures supported on the \pcs{} is affine homeomorphic to the space of invariant probability measures of the generalized odometer associated to~$Q$ (Theorem~\ref{t:measure isomorphism}).
This last space was introduced in~\cite{BruKelsPi97}; we recall its definition~\S\ref{ss:kneading odometer} and in~\S\ref{ss:unique ergodicity} we show that in certain cases this system is uniquely ergodic.
The statement and proof of Theorem~\ref{t:measure isomorphism} is in~\S\ref{ss:measure isomorphism}.
See~\cite{BarDowLia02,GraLiaTic95} for background on generalized odometers.
\subsection{The generalized odometer associated to a kneading map}\label{ss:kneading odometer}
Let $Q : \NN \to \NN$ be a kneading map and put
\begin{multline*}
\Omega_Q \= \{ (x_k)_{k \ge 0} \in \{ 0, 1 \}^\NN \mid x_k = 1 \text{ implies that for each}
\\
j = Q(k + 1), \ldots, k - 1 \text{ we have $x_j = 0$} \}.
\end{multline*}
If we denote by $(S_k)_{k \ge 0}$ the strictly increasing sequence of positive integers defined recursively by $S_0 = 1$ and $S_k = S_{k - 1} + S_{Q(k)}$,
it can be shown that for each non-negative integer~$n$ there is a unique sequence $\expansion{n} \= (x_k)_{k \ge 0}$ in~$\Omega_Q$, that has at most finitely many~$1$'s, and such that~$\sum_{k \ge 0} x_k S_k = n$.
The sequence $\expansion{n}$ is also characterized as the unique sequence in $\{0, 1 \}^\NN$ with finitely many~$1$'s such that $\sum_{k \ge 0} x_k S_k = n$, and that it is minimal with this property with respect to the lexicographical order in~$\{0, 1 \}^\NN$.

When~$Q$ diverges the map defined on the subset~$ \{ \expansion{n} \mid n \in \NN \}$ of~$\Omega_Q$ by $\expansion{n} \mapsto \expansion{n + 1}$, extends continuously to a map~$T_Q : \Omega_Q \to \Omega_Q$ which is onto, minimal, and such that~$T_Q^{-1}$ is well defined on $\Omega_Q \setminus \expansion{0}$; see~\cite[Lemma~2]{BruKelsPi97}.
We call $(\Omega_Q, T_Q)$ the \textit{generalized odometer} associated to~$Q$.
Given $x = (x_k)_{k \ge 0} \in \Omega_Q$ and an integer $n \ge 0$, put $\sigma(x|n) = \sum_{k = 0}^nx_kS_k$.
Observe that $\sigma(x|n)$ is non-decreasing with~$n$, and when~$x$ has infinitely many~$1$'s, $\sigma(x|n) \to + \infty$ as $n \to + \infty$.
On the other hand, if~$x$ has at most a finite number of~$1$'s, then $\sigma(x) \= \lim_{n \to + \infty} \sigma(x|n)$ is finite and $x = \expansion{\sigma(x)}$.

For $x = (x_k)_{k \ge 0}$ different from $\expansion{0}$ we denote by~$q(x) \ge 0$ the least integer such that $x_{q(x)} \neq 0$. 
In~\cite[Theorem~1]{BruKelsPi97} it is shown that if $\parameter \in (0, 4]$ is a parameter such that the kneading map of the logistic map~$f_\parameter$ is equal to~$Q$, then for each $x \in \Omega_Q$ with infinitely many~$1$'s the sequence of intervals $(D_{\sigma(x|n)})_{n \ge q(x)}$ is nested and that $\bigcap_{n \ge q(x)} D_{\sigma(x|n)}$ is reduced to a point belonging to the \pcs~$X_{f_\parameter}$ of~$f_\parameter$.
Furthermore, if we denote this point by~$\pi(x)$ and for $n \ge 0$ we put $\pi(\expansion{n}) = f_\parameter^n(c)$, then the map $\pi : \Omega_Q \to X_{f_\parameter}$ so defined is continuous and conjugates the action of~$T_Q$ on~$\Omega_Q$, to the action of~$f_\parameter$ on~$X_{f_\parameter}$.
\subsection{Generalized odometers as odometers}\label{ss:unique ergodicity}
Let $(p_j)_{j \ge 0}$ be an increasing sequence of positive integers such that for each $j \ge 0$ we have $p_j | p_{j + 1}$.
For each $j \ge 0$ let $\tau_j : \ZZ/p_{j + 1} \ZZ \to \ZZ/p_j \ZZ$ be the reduction map, and $T_j : \ZZ / p_j \ZZ \to \ZZ / p_j \ZZ$ the translation by~1.
Then the \textit{odometer} associated to the sequence~$(p_j)_{j \ge 0}$ is by definition the map~$T$ acting on the inverse limit $\varprojlim_{j} \left( {\ZZ} / p_j {\ZZ}, \tau_j \right)$, that is defined by $T \left( (x_j)_{j \ge 0} \right) = \left( T_j(x_j) \right)_{j \ge 0}$.
It is well-known that each odometer is uniquely ergodic.
\begin{lemma}\label{l:unique ergodicity}
Let~$Q$ be a non-decreasing kneading map such that there is an increasing sequence of positive integers~$(k_j)_{j \ge 0}$ such that for each $j \ge 0$ we have $Q(k_j + 1) = k_j$.
Then the following properties hold.
\begin{enumerate}
\item[1.]
For each $j \ge 0$ and $k \ge k_j$ the integer~$S_{k_j}$ divides $S_{k}$.
In particular~$S_{k_j}$ divides~$S_{k_{j + 1}}$.
\item[2.]
The generalized odometer $(\Omega_Q, T_Q)$ is topologically conjugated to the odometer associated to the sequence~$(S_{k_j})_{j \ge 0}$.
In particular~$(\Omega_Q, T_Q)$ is uniquely ergodic.
\end{enumerate}
\end{lemma}
\begin{proof}
Let~$(\tau_j)_{j \ge 0}$, $(T_j)_{j \ge 0}$ and~$T$ be the maps defined above the statement of the lemma when $(p_j)_{j \ge 0} \= (S_{k_j})_{j \ge 0}$.

\partn{1}
Given~$j \ge 0$ we proceed by induction in~$k \ge k_j$.
The case $k = k_j$ being trivial, we suppose that the integer $k \ge k_j$ is such that for each $k' = k_j, \ldots, k$ the integer~$S_{k_j}$ divides~$S_{k'}$.
As $Q(k + 1) \le k$ and $Q(k + 1) \ge Q(k_j + 1) = k_j$, it follows that~$S_{k_j}$ divides~$S_{Q(k + 1)}$ and $S_{k + 1} = S_k + S_{Q(k + 1)}$.

\partn{2}
For each $j \ge 0$ define the map $\pi_j:\Omega_Q \to \ZZ/S_{k_j}\ZZ$ by
$$
\pi_j((e_n)_{n\geq 0})=e_0S_0+\cdots +e_{k_j-1}S_{k_j-1} \mod S_{k_j}.
$$
By part~1 it follows that for each $j \ge 0$ we have $\tau_j \circ \pi_{j + 1} = \pi_j$.
On the other hand, from the definition of~$\Omega_Q$ we have $\pi_j\circ T_Q = T_j\circ\pi_j$ (see also~\cite[Lemma~3]{BruKelsPi97}.)
Therefore the map $\pi : \Omega_Q \to \varprojlim_{j} (\ZZ/p_j \ZZ, \tau_j)$ defined by $x \mapsto (\pi_j(x))_{j \ge 0}$ satisfies $\pi \circ T_Q = T \circ \pi$.
The map~$\pi$ is clearly continuous and onto.
To show that~$\pi$ is injective just observe that, from the definition of~$\Omega_Q$ it follows that, if for some $j \ge 0$ and $(e_n)_{n \ge 0}, (e_n')_{n \ge 0} \in \Omega_Q$ we have $\pi_j \left( (e_n)_{n \ge 0} \right) = \pi_j \left( (e_n')_{n \ge 0} \right)$, then for each $n = 0, \ldots, k_j - 1$ we have $e_n = e_n'$.
\end{proof}
\subsection{From the generalized odometer to the \pcs}\label{ss:measure isomorphism}
The purpose of this section is to prove the following theorem.
\begin{theoalph}\label{t:measure isomorphism}
Let~$Q$ be a diverging and non-decreasing kneading map and let $(\Omega_Q, T_Q)$ be the corresponding generalized odometer.
Let~$f$ be a unimodal map whose kneading map is equal to~$Q$, and denote by~$X_f$ its \pcs.
Then the space of invariant probability measures of~$(X_f, f|_{X_f})$ is affine homeomorphic to that of~$(\Omega_Q, T_Q)$.
\end{theoalph}

The proof is based on the following proposition.
\begin{proposition}
Let~$\parameter \in (0, 4]$ be a parameter such that the kneading map~$Q$ of the logistic map~$f_\parameter$ diverges, is non-decreasing, and such that for each sufficiently large~$k \ge 0$ we have
 \begin{equation}\label{11}
Q(k + 1) \ge Q(Q(Q(k))+1) + 2.
\end{equation}
Denote by~$X_{f_\parameter}$ the \pcs{} of~$f_\parameter$, by $(\Omega_Q, T_Q)$ the generalized odometer associated to~$Q$, and by $\pi : \Omega_Q \to X_{f_\parameter}$ the projection defined in~\S\ref{ss:kneading odometer}.
Then each point in~$X_{f_\parameter}$ that is not in the backward orbit of the critical point of~$f_\parameter$ has a unique preimage by~$\pi$.
Furthermore, $\pi$ induces a linear homeomorphism between the space of invariant probability measures of~$(\Omega_Q, T_Q)$, and that of~$(f_\parameter|_{X_{f_\parameter}}, X_{f_\parameter})$.
\label{p:almost isomorphism}
\end{proposition}
\begin{remark}\label{r:simple 11}
A diverging and non-decreasing kneading map~$Q$ such that for each sufficiently large $k \ge 2$ we have $Q(k) \le k - 2$, satisfies inequality~\eqref{11} for each sufficiently large~$k$.
In fact, let~$Q$ be such a map and let $(q_r)_{r \ge 0}$ be the strictly increasing sequence of integers defined by $ Q(\NN) = \{ q_r \mid r \ge 0 \} $.
Fix $r \ge 2$ sufficiently large so that for each $k \ge q_{r - 2}$ we have $Q(k) \le k - 2$ and fix~$k \ge 1$ such that~$Q(k) = q_r$.
Then we have $Q(q_r) \le q_r - 2$, so $Q(q_r) \le q_{r - 1}$,
$$ Q(Q(Q(k)) + 1) \le Q(q_{r - 1} + 1) \le q_{r - 1} - 1, $$
and
$$ Q(Q(Q(k)) + 1) \le q_{r - 2} \le q_r - 2 \le Q(k + 1) - 2. $$
\end{remark}
The proof of Proposition~\ref{p:almost isomorphism} is below.
We will first prove Theorem~\ref{t:measure isomorphism} assuming this proposition.
\begin{proof}[Proof of Theorem~\ref{t:measure isomorphism} given Proposition~\ref{p:almost isomorphism}]
Since the logistic family is full there is a parameter $\parameter \in (0, 4]$ such that the kneading map of the logistic map~$f_\parameter$ is the same as that of~$f$.
In view of Proposition~\ref{p:continuous unimodal}, it is enough to prove the desired assertion of~$f_\parameter$ instead of~$f$.

If for each sufficiently large~$k \ge 0$ we have~$Q(k) \le k - 2$, then the assertion is given by Proposition~\ref{p:almost isomorphism} and Remark~\ref{r:simple 11}.
If this is not satisfied, then there are infinitely many integers $k \ge 0$ such that $Q(k) = k - 1$, so Lemma~\ref{l:unique ergodicity} implies that the generalized odometer $(\Omega_Q, T_Q)$ is uniquely ergodic.
It thus follows that the action of~$f_\parameter$ on its \pcs{} is uniquely ergodic as well.
So the assertion of theorem is also satisfied in this case.
\end{proof}

The rest of this section is devoted to the proof of Proposition~\ref{p:almost isomorphism}.
Let~$Q$ be a diverging kneading map and let $(\Omega_Q, T_Q)$ be the corresponding generalized odometer.
Recall that $T_Q^{-1}$ is well defined on $\Omega_Q \setminus \expansion{0}$.
So, if we denote by $\mathcal{O}(\expansion{0})$ the grand orbit of $\expansion{0}$, then
$$ T^{-1}(\Omega_Q \setminus \mathcal{O}(\expansion{0}))
=
\Omega_Q \setminus \mathcal{O}(\expansion{0}), $$
and all negative iterates of~$T_Q$ are well defined on $\Omega_Q \setminus \mathcal{O}(\expansion{0})$.

The proof of Proposition~\ref{p:almost isomorphism} is based on Lemma~\ref{l:injectivity criterion} and Lemma~\ref{l:criterion satisfied} below.
Recall that for $x = (x_k)_{k \ge 0} \in \Omega_Q$ different from~$\expansion{0}$, we denote by~$q(x)$ the least integer $k \ge 0$ such that $x_k = 1$.
The proof of the following lemma is based on some results and ideas of~\cite{Bru99}.
\begin{lemma}
\label{l:injectivity criterion}
Let~$\parameter \in (0, 4]$ be a parameter such that the kneading map~$Q$ of the logistic map~$f_\parameter$ diverges and such that for every sufficiently large integer $k \ge 0$ inequality~\eqref{11} is satisfied.
Then there is a constant~$K > 0$ such that for every pair of distinct points~$x, x' \in \Omega_Q \setminus \mathcal{O}(\expansion{0})$ that satisfy
$$ \max\{q(x), q(x'))\} \ge K
\text{ and }
Q(q(x) + 1) \neq Q(q(x') + 1), $$
we have $\pi(x)\neq \pi(x')$.
\end{lemma}

The proof of this lemma is based on the following lemma of~\cite{Bru99}.
Denote by~$c$ the critical point of~$f_\parameter$ and for a given $k \ge 0$ denote by~$z_k$ (resp. $\hat{z}_k$) the infimum (resp. supremum) of those points~$z \in [0, c]$ (resp. $z \in [c, 1]$) such that~$f_\parameter^{S_{k + 1}}$ is injective on~$[z, c]$ (resp. $[c, z]$).

\begin{lemma}[\cite{Bru99}, Lemma~4]
\label{bruin}
Let~$\parameter$ and~$Q$ be as in Lemma~\ref{l:injectivity criterion}.
Then there exists a constant~$K > 0$ such that for each $k \ge K$, for each $r \ge 0$ and each $t > r$ satisfying $Q(t + 1) \ge r + 1$, the interior of $D_{S_r + S_t}$ cannot contain~$c_{S_k}$ and at the same time intersect $\{z_{Q(k+1)-1},
\hat{z}_{Q(k+1)-1}\}$.
\end{lemma}
We will use the fact, shown for example in~\cite{Bru95} or~\cite[p.~1272]{BruKelsPi97}, that for $r \ge 0$ we have
  \begin{equation}
  \label{falta1}
  f_\parameter^{S_{r}}(c) \in [z_{Q(r+1)-1}, z_{Q(r+1)}]\cup [\hat{z}_{Q(r+1)},\hat{z}_{Q(r+1)-1}].
  \end{equation}
\begin{proof}[Proof of Lemma~\ref{l:injectivity criterion}]
Put $r = q(x)$ (resp. $r' = q(x')$), $x = (x_k)_{k \ge 0}$, and let~$t > r$ be the least integer~$q$ such that $x_q = 1$ (resp. $x_q' = 1$).
From the definition of~$\Omega_Q$ it follows that
$$ S_t < S_r + S_t < S_{t + 1} \text{ and that } Q(t + 1) > r. $$
In particular $S_r + S_t$ is not a cutting time, and~$r$ and~$t$ satisfy the hypothesis of Lemma~\ref{bruin}.

Assume without loss of generality that $r = q(x) >  q(x') = r'$, so that $Q(r' + 1) > Q(r + 1)$.
We will show that the intervals~$D_{S_{r'}}$ and~$D_{S_r + S_t}$ are disjoint, which implies that $\pi(x) \neq \pi(x')$.

For each $n \ge 1$ put $c_n = f_\parameter^n(c)$.

\partn{1}
We will show first that $c_{S_{r'}} \not \in D_{S_r + S_t}$.
Assume by contradiction that this is not the case.
Since $S_t < S_r + S_t < S_{t + 1}$, it follows that $D_{S_r + S_t} = [c_{S_r + S_t}, c_{S_r}]$, and that $c_{S_{r'}} \neq c_{S_r + S_t}$.
As $r' >  r$ we have $c_{S_{r'}} \neq c_{S_r}$, so~$c_{S_{r'}}$ must belong to the interior of~$D_{S_r + S_t}$.
By Lemma~\ref{bruin} with $k = r' \ge K$, to get a contradiction we just need to show that~$D_{S_r + S_t}$ intersects $\{ z_{Q(r' + 1) - 1}, \hat{z}_{Q(r' + 1) - 1} \}$.

By~\eqref{falta1} we have
\begin{equation}\label{e:outside}
c_{S_r} \not \in [z_{Q(r + 1)}, \hat{z}_{Q(r + 1)}].
\end{equation}
Using $Q(r' + 1) > Q(r + 1)$ and using~\eqref{falta1} again, but with~$r$ replaced by~$r'$, we obtain
\begin{equation}\label{e:inside}
c_{S_{r'}}
\in
(z_{Q(r'+1)-1}, \hat{z}_{Q(r'+1)-1})
\subset
(z_{Q(r + 1)}, \hat{z}_{Q(r + 1)}).
\end{equation}
Combined with~\eqref{e:outside} this implies that~$D_{S_r + S_t}$ intersects $\{ z_{Q(r' + 1) - 1}, \hat{z}_{Q(r' + 1) - 1} \}$.

\partn{2}
To complete the proof of the lemma, observe that first that~\eqref{e:inside} implies $D_{S_{r'}} \subset (z_{Q(r + 1)}, \hat{z}_{Q(r + 1)}).$
Therefore~\eqref{e:outside} implies that $D_{S_r + S_t}$ is not contained in~$D_{S_{r'}}$.
So to prove that~$D_{S_r + S_t}$ and~$D_{S_{r'}}$ are disjoint we just need to show that~$D_{S_r + S_t}$ is disjoint from $\partial D_{S_{r'}} = \{ c_{S_{r'}}, c \}$.
We showed $c_{S_{r'}} \not \in D_{S_r + S_t}$ in part~1, and $c \not \in D_{S_r + S_t}$ follows from the fact that $S_r + S_t$ is not a cutting time.
\end{proof}
\begin{lemma}\label{l:criterion satisfied}
Let~$Q : \NN \to \NN$ be a non\nobreakdash-decreasing and diverging kneading map, and let $(\Omega_Q, T_Q)$ be the corresponding generalized odometer.
Then for each constant $K > 0$, and for every pair of distinct points $x, x'$ in~$\Omega_Q$ that are not in the grand orbit of~$\expansion{0}$, there is an integer~$m$ satisfying
$$ \max \{ q(T_Q^m(x)), q(T_Q^m(x')) \} \ge K \text{ and } Q(q(T_Q^m(x)) + 1) \neq Q(q(T_Q^m(x')) + 1). $$
\end{lemma}
\begin{proof}
Let $K > 0$ be given.

\partn{1}
We will show that there is an integer~$m'$ such that
$$ \max \{ q(T_Q^{m'}(x)), q(T_Q^{m'}(x')) \}
\ge
K \text{ and } q(T_Q^{m'}(x)) \neq q(T_Q^{m'}(x')). $$
Let $m'' \ge 0$ be such that~$q(T_Q^{m''}(x)) \ge K$.
Replacing~$x$ and~$x'$ by~$T_Q^{m''}(x)$ and $T_Q^{m''}(x')$, respectively, if necessary, we assume that $\max \{ q(x), q(x') \} \ge K$. 

If $q(x) \neq q(x')$ then take $m' = 0$.
Suppose that $q(x) = q(x')$ and put $x = (x_k)_{k \ge 0}$, $x' = (x_k')_{k \ge 0}$.
As~$x$ and~$x'$ are different there is an integer $k \ge 0$ such that $x_{k} \neq x_k'$.
We denote by~$k_0$ the least integer with this property; we have $k_0 > q(x) = q(x')$.
Put $m' \= - \sum_{\ell = 0}^{k_0 - 1} x_\ell S_\ell = - \sum_{\ell = 0}^{k_0 - 1} x_\ell' S_\ell$.
Then it follows that $( \hat{x}_k )_{k \ge 0} \= T_Q^{m'}(x)$ (resp. $( \hat{x}_k' )_{k \ge 0} \= T_Q^{m'}(x')$) is such that for each $\ell \in \{0, \ldots, k_0 - 1 \}$ we have $\hat{x}_\ell = 0$ (resp. $\hat{x}_\ell' = 0$) and for every $\ell \ge k_0$ we have $\hat{x}_\ell = x_\ell$ (resp. $\hat{x}_\ell' = x_\ell'$).
Thus $q(T_Q^{m'}(x)) \neq q(T_Q^{m'}(x'))$ and
$$ \max \{ q(T_Q^{m'}(x)), q(T_Q^{m'}(x')) \}
\ge k_0 > q(x) = q(x') \ge K. $$

\partn{2}
Let $m'$ be the integer given by part~1, and put $y \= T_Q^{m'}(x)$ and $y' \= T_Q^{m'}(x')$.
Assume without loss of generality that $q(y) < q(y')$, so that $q(y') \ge K$.
If $Q(q(y) + 1) \neq Q(q(y') + 1)$ then take $m = m'$.
So we assume that $q_0 \= Q(q(y) + 1) = Q(q(y') + 1)$ and put $Q^{-1}(q_0) = \{ k_1 + 1, \ldots, k_1' \},$ so that $q(y), q(y') \in \{ k_1, \ldots, k_1' - 1\}$.
Note that $k_1' - 1 \ge q(y') \ge K$, so that $k_1' > K$.

Let us show that for each $k \in \{ q(y) + 1, \ldots, k_1' - 1\}$ we have $y_k = 0$.
In fact, suppose by contradiction that there is such~$k$ satisfying $y_k = 1$.
Then the chain of inequalities
$$ Q(k + 1) \le Q(k_1') = q_0 = Q(k_1 + 1) \le k_1 \le q(y), $$
and the definition of~$\Omega_Q$ imply that $y_{q(y)} = 0$.
This contradiction proves the claim.

Therefore $( \hat{y}_k )_{k \ge 0} \= T_Q^{-S_{q(y)}}(y)$ is such that for all $k \in \{ 0, \ldots, k_1' - 1 \}$ we have $\hat{y}_k = 0$.
Since~$y$ is not in the grand orbit of~$\expansion{0}$ this implies that $q(T_Q^{-S_{q(y)}}(y)) \ge k_1' > K$ and that $Q(q(T_Q^{-S_{q(y)}}(y)) + 1) > q_0$.
On the other hand, since $q(y) < q(y')$, we have $q(T_Q^{-S_{q(y)}}(y')) \le q(y') - 1$, so
$$ Q(q(T_Q^{-S_{q(y)}}(y')) + 1) \le Q(q(y')) \le Q(k_1') = q_0. $$
This shows that the integer $m = m' - S_{q(y)}$ satisfies the desired properties.
\end{proof}

\begin{proof}[Proof of Proposition~\ref{p:almost isomorphism}]
To prove that each point in~$X_{f_\parameter}$ that is not in the backward preimage of the critical point~$c$ of~$f_\parameter$ has a unique preimage by~$\pi$, observe first that Lemma~\ref{l:injectivity criterion} and Lemma~\ref{l:criterion satisfied} imply that~$\pi$ is injective on~$\Omega_Q \setminus \mathcal{O}(\expansion{0})$.
We will use the fact, shown in the proof of~\cite[Theorem~1]{BruKelsPi97}, that $\pi^{-1}(c) = \{ \expansion{0} \}$.
This implies that each point in the forward orbit of~$c$ has a unique preimage by~$\pi$, and that the preimage by~$\pi$ of the grand orbit of~$c$ is equal to the grand orbit of~$\expansion{0}$.
So the preimage by~$\pi$ of a point outside the grand orbit of~$c$ is contained in~$\Omega_Q \setminus \mathcal{O}(\expansion{0})$.
As~$\pi$ is injective on this set it follows that each point outside the grand orbit of~$c$ has at most one preimage by~$\pi$.

To prove the last assertion of the proposition we just need to show that an invariant measure of~$T_Q$ cannot charge~$\mathcal{O}(\expansion{0})$.
Since~$X_{f_\parameter}$ is a Cantor set and~$f_\parameter$ is minimal on~$X_{f_\parameter}$ (Proposition~\ref{p:continuous unimodal}), it follows that the forward orbit by~$f_{\parameter}$ of each point in~$X_{f_\parameter}$ is infinite.
Hence the forward orbit by~$T_Q$ of each point in~$\Omega_Q$ is infinite.
It follows that an invariant measure of~$T_Q$ cannot charge points, and in particular that such a measure cannot charge~$\mathcal{O}(\expansion{0})$.
\end{proof}

\section{The Bratteli-Vershik system associated to a kneading map}
\label{s:Bratteli-Vershik system}
The purpose of this section is to recall the definition of Bratteli-Vershik system associated to a kneading map, that was introduced by Bruin in~\cite{Bru03}.
After briefly recalling the concepts of Bratteli diagram (\S\ref{ss:diagram}) and Bratteli-Vershik system (\S\ref{ss:Bratteli-Vershik}), we give a representation of the corresponding invariant measures as an inverse limit of linear maps (\S\ref{ss:invariant measures}).
We define the Bratteli-Vershik system associated to a kneading map in~\S\ref{ss:kneading Bratteli-Vershik}.
See for example~\cite{DurHosSka99,HerPutSka92} and references therein for background and further properties of Bratteli-Vershik systems.
\subsection{Bratteli diagrams}\label{ss:diagram}
A \textit{Bratteli diagram} is an infinite directed graph $(V,E)$,
such that the vertex set~$V$ and the edge set~$E$ can be partitioned
into finite sets
$$
V = V_0\cup V_1 \cup  \cdots \mbox{ and } E=E_1\cup E_2\cup
\cdots
$$
with the following properties:
\begin{itemize}
\item
$V_0=\{v_0\}$ is a singleton.
\item
For every $j \ge 1$, each edge in~$E_j$ starts in a vertex in~$V_{j - 1}$ and arrives to a vertex in~$V_{j}$.
\item
All vertices in~$V$ have at least one edge starting from it, and all vertices except~$v_0$ have at least one edge arriving to it.
\end{itemize}
For a vertex $e \in E$ we will denote by~$s(e)$ the vertex where~$e$ starts and by~$r(e)$ the vertex to which~$e$ arrives.
A \textit{path} in $(V, E)$ is by definition a finite (resp. infinite) sequence $e_1e_2 \ldots e_j$ (resp. $e_1e_2 ...$) such that for each $\ell = 1, \ldots, j - 1$ (resp. $\ell = 1, \ldots$) we have $r(e_\ell) = s(e_{\ell + 1})$.
Note that for each vertex~$v$ distinct from $v_0$ there is at least one path starting at~$v_0$ and arriving to~$v$.

An \textit{ordered Bratteli diagram} $(V,E,\geq)$ is a Bratteli diagram $(V,E)$ together with a partial order~$\geq$ on~$E$, so that two edges are comparable if and only if they arrive at the same vertex.
For each $j \ge 1$ and $v \in V_j$ the partial order~$\ge$ induces an order on the set of paths from~$v_0$ to~$V$ as follows:
$$
e_1\cdots e_j > f_1\cdots f_j
$$
if and only there exists $j_0 \in \{1,\cdots, j \}$ such that $e_{j_0} >
f_{j_0}$ and such that for each $\ell \in \{ j_0 + 1, \ldots, j \}$ we have $e_\ell = f_\ell$.

We will say that an edge~$e$ is \textit{maximal} (resp. \textit{minimal}) if it is maximal (resp. minimal) with respect to the order~$\ge$ on the set of all edges in~$E$ arriving at~$r(e)$.
Note that for each vertex~$v$ distinct from~$v_0$ there is precisely one path starting at~$v_0$ and arriving to~$v$ that is maximal (resp. minimal) with respect to the order~$\ge$.
It is characterized as the unique path starting at~$v_0$ and arriving at~$v$ consisting of maximal (resp. minimal) edges.

\subsection{Bratteli-Vershik system}\label{ss:Bratteli-Vershik}
Fix an ordered Bratteli diagram $B \=(V,E,\geq)$.
We denote by~$X_B$ set of all infinite paths in~$B$ starting at~$v_0$.
For a finite path $e_1 \ldots e_j$ starting at~$v_0$ we denote by $U(e_1 \ldots e_j)$ the subset of~$X_B$ of all infinite paths $e_1'e_2' \ldots$ such that for all $\ell = 1, \ldots, j$ we have $e_\ell' = e_\ell$.
We endow~$X_B$ with the topology generated by the sets $U(e_1 \ldots e_j)$.
Then each of this sets is clopen, so $X_B$ becomes a compact Hausdorff space with a countable basis of clopen sets.

We will denote by~$X_B^{\max}$ (resp. $X_B^{\min}$) the set of all elements $(e_j)_{j \ge 1}$ of~$X_B$ so that for each $j \ge 1$ the edge~$e_j$ is a maximal (resp. minimal).
It is easy to see that each of these sets is non-empty.

From now on we assume that the set~$X_B^{\min}$ is reduced to a unique point, that we will denote by~$x_{\min}$.
We will then define the transformation $V_B:X_B\to X_B$ as follows:
\begin{itemize}
\item
$V_B^{-1}(x_{\min}) = X_{\max}$.
\item
Given $x \in X_B \setminus X_{\max}$, let $j \ge 1$ be the smallest integer such that~$e_j$ is not maximal.
Then we denote by~$f_j$ the successor of~$e_j$ and by $f_1\ldots f_{j - 1}$ the unique minimal path starting at~$v_0$ and arriving to~$s(f_k)$.
Then we put,
$$V_B(x)=f_1\cdots f_{k-1}f_ke_{k+1}e_{k+2}\ldots \ .$$
\end{itemize}
The map~$V_B$ is continuous, onto and invertible except at~$x_{\min}$.
\subsection{Transition matrices and invariant measures}\label{ss:invariant measures}
We fix an ordered Bratteli diagram $B \= (V, E, \ge)$ having a unique minimal infinite path, and consider the map $V_B : X_B \to X_B$ defined in the previous subsection.

For $j \ge 1$ and $v \in V_j$ we denote by~$s_j(v) > 0$ the number of paths starting at~$v_0$ and arriving to~$v$, and put $\vec{s}_j \= ( s_j(v) )_{v \in V_j} \in \RR^{V_j}$.
Let $N_j \in \M_{V_{j - 1}, V_j}$ be the matrix such that for each $v \in V_{j - 1}$ and $v' \in V_{j}$ the entry $N_j(v, v')$ is equal to the number of edges starting at~$v$ and arriving to~$v'$.
Observe that $N_j^t \vec{s}_{j - 1} = \vec{s}_{j}$, so if we put $B_0 = \{ 1 \} \in \M_{V_0, V_0}$ and for each $j \ge 1$ we denote by $B_j \in \M_{V_j, V_j}$ the diagonal matrix defined by $B_j(v, v) = s_j(v)$, then the matrix
$$ M_j \= B_{j - 1} N_j B_j^{-1} \in \M_{V_{j - 1}, V_{j}}, $$
is stochastic.  

Given $j \ge 1$ and $v \in V_j$, denote by~$\underline{e}_v$ be the maximal path starting at~$v_0$ and arriving to~$v$.
Then we put $C_v \= U(\underline{e}_v)$, which is a clopen subset of~$X_B$.
If for a given $s \in \{ 0, \ldots,  s_j(v) - 1 \}$ we denote by $f_1\cdots f_{j-1}$ the $(s_j(v) - s)$\nobreakdash-th path from~$v_0$ to~$v$, then we have $V_B^{-s}(C_v) = U(f_1,\cdots, f_{j-1})$.
It thus follows that
\begin{eqnarray*}
\P_j
& \= &
\{V_B^{-s}(C_v)\mid s \in \{0, \ldots, s_j(k) - 1 \}, k \in V_j \}
\\ & = &
\{ U(e_1 \ldots e_{j - 1}) \mid e_1 \ldots e_{j - 1} \text{ path starting at } v_0 \}.
\end{eqnarray*}
is a partition of~$X_B$ into clopen sets.
Note also that the partition~$\P_{j+1}$ is finer than~$\P_j$.
It follows that each invariant measure~$\mu$ of $(X_B,V_B)$ is determined by its values on the sets belonging to the partitions~$\P_j$.

We will need the following fact for the proof of Lemma~\ref{l:medidas Bratteli} below: If $(X_B, V_B)$ has no periodic points, then
 \begin{equation}
 \label{alturas}
\lim_{n\to + \infty}\min\{s_n(v): v\in V_n\} = + \infty.
\end{equation}
Indeed, if~\eqref{alturas} does not hold, then there exist $N \in \NN$ and a strictly increasing sequence of positive integers $(m_n)_{n \ge 1}$ such that for each $n \ge 1$,
$$
\min\{s_{m_n}(v): v\in V_n\} = N.
$$
By compactness this implies there exists a sequence $(x_k)_{k\geq 1}$ in~$X_B$ converging to~$x_{\min}$, such that if for each $k \ge 1$ there is $n_k \ge 1$ and $v_k \in V_{n_k}$ such that $s_{m_{n_k}}(v_k) = N$ and $x_k \in V_B^{-(N-1)}(C_{v_{k}})$.
It then follows that
$V_B^{N-1}(x_{\min})\in X_{\max}$ and $V_B^N(x_{\min})=x_{\min}$.

The following result is well-known, but we were not able to find an exact reference.
We include a proof for completeness.
Analogous results can be found for example in~\cite[Proposition 3.2]{HerPutSka92} and~\cite[Lemma~3.1 and Proposition~3.2]{GamMar06}.
Recall that for a finite set $V$ we denote by~$\Delta_V$ the unit simplex in~$\RR^V$.
\begin{lemma}\label{l:medidas Bratteli}
Denote by $\M(X_B, V_B)$ the space of signed measures on~$X_B$ that are invariant by~$V_B$, endowed with the weak$^*$ topology.
Furthermore, denote by $\M_1(X_B, V_B)$ the set of probability measures in $\M(X_B, V_B)$.
If $(X_B, V_B)$ has no periodic points, then the linear map~$H : \M(X_B, V_B) \to \varprojlim_{j} (\RR^{V_j}, M_j)$, defined by
$$ \mu \mapsto \{ (s_j(v)\mu(C_v))_{v \in V_j} \}_{j\geq 1}, $$
is a homeomorphism and
$$ H (\M_1(X_B, V_B)) = \varprojlim_{j} (\Delta_{V_j}, M_j). $$
\end{lemma}
\begin{proof}
By definition of the matrices~$M_j$, the inverse limit $\varprojlim_{j}(\RR^{V_j}, M_j)$ is isomorphic to $\varprojlim_{j} (\RR^{V_j}, N_j)$.
It is thus enough to prove the lemma with~$M_j$ replaced by~$N_j$, with $H$ replaced by the map $\widetilde{H} : \M(X_B, V_B) \to \varprojlim_{j} (\RR^{V_j}, N_j)$ defined by
$$ \mu \mapsto \{ (\mu(C_v))_{v \in V_j} \}_{j\geq 1}, $$
and with $\Delta_{V_j}$ replaced by
$$ \widetilde{\Delta}_j \= B_j^{-1} (\Delta_{V_j}) = \left\{ (x_v)_{v \in V_j} \in \RR^{V_j} \mid \sum_{v \in V_j} s_j(v) x_v = 1 \right\}. $$

Keeping the notation above the statement of the lemma, we have for each $j \ge 1$ and $v \in V_j,$
\begin{eqnarray}
\label{e:conjunto}
 C_v & = & \bigcup_{\footnotesize \begin{array}{c}
           e\in E_j \\
           s(e)=v
         \end{array}} U(\underline{e}_v e)\\
        & =  &
\bigcup_{v' \in V_{j + 1}} \bigcup_{\footnotesize \begin{array}{c}
           e \in E_j \\
           s(e)=v \\
           r(e)=v'
         \end{array}} U(\underline{e}_v e)
\end{eqnarray}
Observe furthermore that for each $v' \in V_{j + 1}$ and each $e \in E_j$ such that $s(e) = v$ and $r(e) = v'$, there is $s \in \{ 0, \ldots, s_{j + 1}(v') - 1 \}$ such that $U(\underline{e}_v e) = V_B^{-s}(C_{v'})$.

It thus follows that for each $\{ (x_v)_{v \in V_j} \}_{j \geq 1}$ in
$\varprojlim_{j}(\RR^{V_j}, N_j)$ there is a signed measure $\mu$ on~$X_B$ such that for each $j \ge 1$, $v \in V_j$ and $s \in \{ 0, \ldots, s_j(v) - 1 \}$ we have
$$
\mu(V_B^{-s}(C_v))=x_v.
$$
Since $(\P_j)_{j \ge 1}$ spans the topology of~$X_B$ this measure
is unique. We will now prove that this measure is invariant
by~$V_B$. It is enough to show that for each $j \ge 1$ and each $P
\in \P_j$ we have $\mu(V_B^{-1}(P)) = \mu(P)$. By definition of
$\P_j$, for each $P \in \P_j$ there is $v \in V_j$ and $s \in \{
0, \ldots, s_j(v) - 1 \}$ such that $P = V_B^{-s}(C_v)$. If $s <
s_j(v) - 1$ then we have $\mu(V_B^{-1}(P)) = \mu(V_B^{- (s +
1)}(C_v)) = x_v = \mu(V_B^{- s}(C_v)).$
Suppose now that $s = s_j(v)-1$.
Since~\eqref{alturas} holds, we can take $m>j$ such that for each $v' \in V_m$ we have $s_m(v') > s_j(v)$.
After some computations, and using the equality,
$$ C_v = \bigcup_{v'\in
V_m}\bigcup_{\footnotesize
\begin{array}{c}
           r = 0, \ldots, s_m(v') - 1,\\
           V_B^{-r}(C_{v'})\subseteq C_v
         \end{array}}V_B^{-r}(C_{v'}),$$
we get
\begin{multline*}
\left| \mu \left( V_B^{-s_j(v)}(C_v) \right) - \mu(C_v) \right|
\le
\mu \left( \bigcup_{v'\in V_m}\bigcup_{r=s_m(v')-s_j(v)}^{s_m(v')-1}V_B^{-r}(C_v') \right)
+ \\ +
\mu \left( \bigcup_{v'\in V_m}\bigcup_{r=0}^{s_j(v)-1}V_B^{-(s_m(v)+r)}(C_v') \right).
\end{multline*}
Since for each $r = 0, \ldots, s_j(v)$ we have $V_B^{-(s_m(v')+r)}(C_{v'})\subseteq \bigcup_{v''\in V_m}V_B^{-r}(C_{v''})$, we obtain
$$
\left| \mu \left( V_B^{-s_j(v)}(C_v) \right) - \mu(C_v) \right|
\le
2s_j(v)\sum_{v'\in V_m}\mu(C_{v'})
\le
\frac{2s_j(v)\mu(X_B)}{\min\{s_m(v'): v'\in V_m\}}.
$$
From (\ref{alturas}) we deduce that
$\mu(V_B^{-s_j(v)}(C_v))=\mu(C_v).$

We have thus shown that~$\widetilde{H}$ has an inverse that is defined on $\varprojlim_{j}(\RR^{V_j}, N_j)$ and that takes images in~$\M(X_B, V_B)$.
As for each $j \ge 1$ all of the elements of $\P_j$ are clopen, it follows that both,~$\widetilde{H}$ and its inverse are continuous.

To complete the proof of the lemma it remains to show that $\widetilde{H}(\M_1(X_B, V_B)) = \varprojlim_{j} (\widetilde{\Delta}_j, N_j)$.
The inclusion $\widetilde{H}^{-1}(\varprojlim_{j} (\widetilde{\Delta}_{j}, N_j)) \subset \M_1(X_B, V_B)$ is easily seen to hold.
By~\eqref{e:conjunto} and by the remark after it, we have for each $\mu \in \M_1(X_B, V_B)$, $j \ge 1$ and $v \in V_j$,
\begin{multline}
\label{m1}
\mu(C_v) =
\sum_{v' \in V_{j + 1}} \sum_{\footnotesize
\begin{array}{c}
           e\in E_j \\
           s(e)=v\\
           r(e)=v'
         \end{array}}\mu(C_{v'})
= \sum_{v' \in V_{j + 1}} N_{j + 1}(v, v') \mu(C_{v'}).
\end{multline}
Furthermore, since $\mu(X_B)=1$ and $\mu$ is invariant, it holds
\begin{equation}
\label{m2} \sum_{v \in V_j} s_j(v)\mu(C_v)=1.
\end{equation}
Equations (\ref{m1}) and (\ref{m2}) imply that $\widetilde{H}(\mu) \in \varprojlim_{j}(\widetilde{\Delta}_j, N_j).$
\end{proof}
\subsection{The Bratteli-Vershik system associated to a kneading map}\label{ss:kneading Bratteli-Vershik}
Given a kneading map~$Q$ we will now define an ordered Bratteli diagram $B_Q \= (V, E, \le)$ that was introduced by Bruin in~\cite[\S4]{Bru03}.

We start defining the Bratteli diagram $(V, E)$:
\begin{itemize}
\item
$V_0 = \{ 0 \}$, $V_1 = \{ k \in \NN \mid k \ge 1, Q(k) = 0 \}$, and for $j \ge 2$,
$$ V_j \= \{ k \in \NN \mid k \ge j, Q(k-1) \le j - 2 \}. $$
\item
$E_1 = \{ (0 \to k) \mid k \in V_1 \}$ and for $j \ge 2$
\begin{multline*}
E_j = \{ j - 1 \to j \} \cup \{ j - 1 \to k \mid k \in V_j \setminus V_{j - 1} \}
\cup
\{ k \to k \mid k \in V_j \cap V_{j - 1} \}.
\end{multline*}
\end{itemize}
Note that for every $j \ge 2$, each vertex in~$V_j$ different from~$j$ has at most one edge arriving at it.
Besides $\{j - 1 \to j \} \in E_j$, the only edge that can arrive to $j \in V_j$ is $\{j \to j \} \in E_j$, that only exists when $j \in V_{j - 1}$.

So to define the partial order~$\ge$, we just have to define it, for each $j \ge 2$, between $\{j - 1 \to j \} \in E_{j - 1}$ and $\{ j \to j \} \in E_{j - 1}$ when both exist: we put $\{ j - 1 \to j \} < \{ j \to j \}$.
The rest of the edges are maximal and minimal at the same time.

Note that for $k \ge 1$ the set~$V_k$ is reduced to a point if and only if~$Q(k) = k - 1$.
So, if for each large $k \ge 1$ we have $Q(k) = k - 1$, then the set~$X_{B_Q}$ is finite.
Otherwise, it follows that the set~$X_{B_Q}$ is a Cantor set.

It is straight forward to check that the infinite path $0 \to 1 \to 2 \to \cdots $ is the unique minimal path in~$B_Q$.
Therefore there is a well defined map $V_{B_Q} : X_{B_Q} \to X_{B_Q}$, see~\S\ref{ss:Bratteli-Vershik}.
The following is \cite[Proposition~2]{Bru03}, and the last statement follows from \cite[Lemma~2]{BruKelsPi97}.
\begin{theorem}[\cite{Bru03}, Proposition~2]\label{t:reduction to Bratteli}
Let~$Q$ be a diverging kneading map, and consider the corresponding Bratteli-Vershik system $(X_{B_Q}, V_{B_Q})$ and generalized odometer $(\Omega_Q, T_Q)$.
Then there is a homeomorphism between~$X_{B_Q}$ and~$\Omega_Q$ that conjugates the action of~$V_{B_Q}$ on $X_{B_Q}$ to the action of~$T_Q$ on~$\Omega_Q$.
In particular $(X_{B_Q}, V_{B_Q})$ is minimal.
\end{theorem}
We will also need the following simple lemma.
\begin{lemma}\label{l:basic Bratteli-Vershik}
Let~$Q$ be a non-decreasing and diverging kneading map and let~$\widetilde{Q}$ be a kneading map such that for each sufficiently large~$k \ge 0$ we have $\widetilde{Q} (k) = Q(k)$.
Then the space of invariant probability measures of~$(X_{B_{\widetilde{Q}}}, V_{B_{\widetilde{Q}}})$ is affine homeomorphic to that of~$(X_{B_Q}, V_{B_Q})$.
\end{lemma}
\begin{proof}
Suppose first that for every large $k \ge 1$ we have $Q(k) = k - 1$, so that~$X_{B_Q}$ is finite.
As $(X_{B_Q}, V_{B_Q})$ is minimal it follows that it is uniquely ergodic.
Using a similar reasoning we obtain that~$(X_{B_{\widetilde{Q}}}, V_{B_{\widetilde{Q}}})$ is also uniquely ergodic, so the lemma holds in this case.

Suppose now that there are infinitely many integers $k \ge 1$ such that $Q(k) \le k - 2$, so that $X_{B_Q}$ is a Cantor set.
As $(X_{B_Q}, V_{B_Q})$ is minimal, it follows that it does not have periodic points, so it satisfies the hypothesis of Lemma~\ref{l:medidas Bratteli}.
Similarly,~$(X_{B_{\widetilde{Q}}}, V_{B_{\widetilde{Q}}})$ also satisfies the hypothesis of Lemma~\ref{l:medidas Bratteli}.

Denote by $\widetilde{V} = \bigcup_{j \ge 0} \widetilde{V}_j$ the set of vertices, $\widetilde{E} = \bigcup_{j \ge 1} \widetilde{E}_j$ the set of edges, and~$\widetilde{\ge}$ the partial order defining~$B_{\widetilde{Q}}$.
Then for each sufficiently large~$j$ we have $\widetilde{V}_j = V_j$, $\widetilde{E}_j = E_j$, and the restriction of~$\widetilde{\ge}$ to~$E_j$ coincides with that of~$\ge$.
It thus follows that the corresponding transition matrices are the same.
Then the desired assertion follows from Lemma~\ref{l:medidas Bratteli}.
\end{proof}

\section{Invariant measures of resonant unimodal maps}
\label{s:invariant of resonant}
The purpose of this section is to prove Theorem~\ref{t:invariant of resonant}, stated in~\S\ref{s:resonant kneading}.
In view of Theorem~\ref{t:measure isomorphism} (\S\ref{ss:measure isomorphism}) and Theorem~\ref{t:reduction to Bratteli} (\S\ref{ss:kneading Bratteli-Vershik}), we just need to prove the analog statement for the corresponding Bratteli-Vershik system defined in~\S\ref{s:Bratteli-Vershik system}.
After calculating the corresponding transition matrices in~\S\ref{ss:transition matrices}, we give the proof of Theorem~\ref{t:invariant of resonant} in~\S\ref{ss:proof of invariant of resonant}.
\subsection{The transition matrices}\label{ss:transition matrices}
Fix a diverging and non-decreasing function $Q : \NN \to \NN$ such that for each $k \ge 2$ we have $Q(k) \le k - 2$.
It is a kneading map, and we consider the ordered Bratteli diagram $B_Q \= (V, E, \ge)$ defined in~\S\ref{ss:kneading Bratteli-Vershik}.

Let $(q_r)_{r \ge 0}$ be the strictly increasing sequence defined by $Q(\NN)=\{q_r \mid \, r \ge 0 \},$ and let $(k_r)_{r \ge 0}$ be the strictly increasing sequence of integers such that for each $r \ge 0$ we have,
$$ Q^{-1}(q_r) = \{ k_r, k_r + 1, \ldots, k_{r + 1} - 1 \}. $$
Note that $q_0 = k_0 = 0$, and $k_1 \ge 3$.
Moreover, for every $r \ge 1$ we have $q_r = Q(k_r) \le k_r - 2$.

From the definition of~$B_Q$ it follows that $V_1 = \{ 1, \ldots, k_1 - 1 \}$, and that for each $r \ge 1$ and $j \in \{ q_{r - 1} + 2, \ldots, q_r + 1 \}$ we have
$$ V_j = \{ j, \ldots, k_{r} \}. $$
In particular $V_2 = \{ 2, \ldots, k_1 \}$.

In the following lemma we use the notation introduced in~\S\ref{ss:invariant measures} for $B = B_Q$.
Part~1 is shown in Remark~1 after the statement of Theorem~2 in~\cite{Bru03}; we provide the short proof for completeness.
Recall that for a finite and non-empty set~$V$ and for $v \in V$ we denote by $\vec{e}_v \in \RR^V$ the vector having all of its coordinates equal to~$0$, except for the coordinate corresponding to~$v$ that is equal to~$1$.
\begin{lemma}
Define $(S_k)_{k \ge0}$ recursively by $S_0 = 1$ and $S_k = S_{k - 1} + S_{Q(k)}$.
Then we have the following properties.
\begin{enumerate}
\item[1.]
For each $j \ge 1$ we have $j + 1 \in V_j$ and $s_j(j) = S_{j - 1}$.
\item[2.]
For each $j = q_{r - 1} + 3, \ldots, q_r + 1$, we have
$$ M_j \in \M_{ \{j - 1, \ldots, k_r \}, \{ j, \ldots, k_r \} }, $$
and
$$
M_j(\cdot, \ell) = \left\{ \begin{array}{ll}
                        \frac{S_{j-2}}{S_{j - 1}} \vec{e}_{j - 1}+ \frac{S_{Q(j - 1)}}{S_{j - 1}}\vec{e}_{j} & \mbox{ if } \ell = j, \\
                         \vec{e}_{\ell} & \mbox{ if } \ell \in \{ j + 1, \ldots, k_{r} \}. \\
                       \end{array}\right. 
$$
Moreover, when $j = q_r + 2$ we have
$$ M_{q_r + 2} \in \M_{ \{ q_r + 1, \ldots, k_r \}, \{ q_r + 2, \ldots, k_{r + 1} \} }, $$
and,
$$
M_{q_r + 2} (\cdot, \ell) = \left\{ \begin{array}{ll}
                        \frac{S_{q_r}}{S_{q_r + 1}} \vec{e}_{q_r + 1}  + \frac{S_{Q(q_r + 1)}}{S_{q_r + 1}}\vec{e}_{q_r + 2} & \mbox{ if } \ell = q_r + 2, \\
                         \vec{e}_\ell & \mbox{ if } \ell \in \{ q_r + 3, \ldots, k_r \}, \\
                         \vec{e}_{q_r + 1}     & \mbox{ if } \ell \in \{ k_r + 1, \ldots, k_{r + 1} \}.
                         \end{array}\right.
$$
\end{enumerate}
\end{lemma}
\begin{proof}
\

\partn{1}
Since $V_1 = \{ 1, \ldots, k_1 - 1 \}$ and $k_1 \ge 3$, we have that $2 \in V_1$.
On the other hand, for each $r \ge 1$ we have $q_r \le k_r - 2$ so for each $j \ge 2$ the set~$V_j$ contains $j + 1$.
In particular for each $j \ge 2$ there are precisely~$2$ edges in~$E_j$ arriving at $j \in V_{j}$.

We will prove by induction that for each $j \ge 1$ we have $s_j(j) = S_{j - 1}$.
This clearly holds for $j = 1, 2$.
Fix $j_0 \ge 2$ and suppose that this equality holds for all $j = 1, \ldots, j_0$.
From the definition of~$B_Q$ we have
$$ s_{j_0 + 1}(j_0 + 1) = s_{j_0}(j_0) + s_{j_0}(j_0 + 1). $$
So we just need to prove that $s_{j_0}(j_0 + 1) = S_{Q(j_0)}$.
If $j_0 \le k_1 - 1$, then $s_{j_0}(j_0 + 1) = s_2(j_0 + 1) = 1 = S_{Q(j_0)}$.
Suppose now that $j_0 \ge k_1$ and let $r \ge 1$ be such that $j_0 \in \{ k_r, \ldots, k_{r + 1} - 1 \}$.
Then $j_0 + 1 \in V_{q_r + 2}$, but $j_0 + 1 \notin V_{q_r + 1}$, so
$$ s_{j_0}(j_0 + 1) = s_{q_r + 2}(j_0 + 1) = s_{q_r + 1}(q_r + 1) = S_{q_r} = S_{Q(j_0)}. $$

\partn{2}
Observe that by definition for each $k \in V_{j - 1}$ and $\ell \in V_{j}$ we have $M_j(k, \ell) = s_{j}(\ell)^{-1} N_j(k, \ell) s_{j - 1}(k)$.

Fix $r \ge 1$.
We first consider the case when $j \in \{ q_{r - 1} + 3, \ldots, q_{r} + 1 \}.$
Then we have $N_j(j - 1, j) = 1$, for each $\ell = j, \ldots, k_{r}$ we have $N_j(\ell, \ell) = 1$, and the rest of the entries of~$N_j$ are all equal to~$0$.
Hence, by the (proof of) part~1 we have
$$ M_j( \cdot, j)
=
\frac{s_{j - 1}(j - 1)}{s_{j}(j)} \vec{e}_{j - 1} + \frac{s_{j - 1}(j)}{s_{j}(j)} \vec{e}_{j}
=
\frac{S_{j-2}}{S_{j - 1}} \vec{e}_{j - 1}+ \frac{S_{Q(j - 1)}}{S_{j - 1}}\vec{e}_{j}, $$
and for each $\ell = j + 1, \ldots, k_{r}$ we have $M( \cdot, \ell) = s_{j}(\ell)^{-1} s_{j - 1}(\ell) \vec{e}_\ell = \vec{e}_\ell$.

We suppose now that $j = q_r + 2$.
Then we have $N_j(j - 1, j) = 1$, for each $\ell = j, \ldots, k_r$ we have $N_j(\ell, \ell) = 1$, for each $\ell = k_r + 1, \ldots, k_{r + 1}$ we have $N_j(j - 1, \ell) = 1$, and the rest of the entries of~$N_j$ are all equal to~$0$.
Hence, by the (proof of) part~1 we have
$$ M_j( \cdot, j)
=
\frac{s_{j - 1}(j - 1)}{s_{j}(j)} \vec{e}_{j - 1} + \frac{s_{j - 1}(j)}{s_{j}(j)} \vec{e}_{j}
=
\frac{S_{j-2}}{S_{j - 1}} \vec{e}_{j - 1}+ \frac{S_{Q(j - 1)}}{S_{j - 1}}\vec{e}_{j}, $$
for each $\ell = j + 1, \ldots, k_{r}$ we have $M( \cdot, \ell) = s_{j}(\ell)^{-1} s_{j - 1}(\ell) \vec{e}_\ell = \vec{e}_\ell$, and for each $\ell = k_r + 1, \ldots, k_{r + 1}$ we have $M(\cdot, \ell) = s_{j}(\ell)^{-1}s_{j - 1}(j - 1) \vec{e}_{j - 1} = \vec{e}_{j - 1}$.
\end{proof}
The following lemma is a direct consequence of the previous one.
We leave the straight forward proof to the reader.
\begin{lemma}
\label{l:multiplicacion}
Given $r \ge 1$ let $r' \ge r$ be the integer such that $k_r \in \{q_{r'} + 2, \ldots, q_{r' + 1} + 1 \}$.
Then the rank of the matrix
$$ M_{q_{r} + 2}\cdots M_{k_{r}} \in \M_{V_{q_r + 1}, V_{k_r}}$$
is equal to $r' - r + 2$.
In fact, if for each~$n \in \{ q_r, \ldots, k_r - 1 \}$ we define the vector $\vec{v}(n) \in \RR^{V_{q_r + 1}}$ by,
$$ \vec{v}(n)
\=
\frac{1}{S_{n}} \left( S_{q_{r}} \vec{e}_{q_r + 1} + \sum_{t = q_r + 1}^{n} S_{Q(t)} \vec{e}_{t + 1} \right),$$
then this matrix is equal to,
$$
\left( \begin{array}{lllllll}
                    \vec{v}(k_r - 1), & \underbrace{ \vec{v}(q_{r}), \cdots, \vec{v}(q_{r})}_{k_{r+1}-k_r }, & \underbrace{ \vec{v}(q_{r + 1}), \cdots, \vec{v}(q_{r + 1})}_{k_{r + 2} - k_{r + 1}}, &  \cdots,   &
                      \underbrace{\vec{v}(q_{r'}), \cdots, \vec{v}(q_{r'})}_{k_{r' + 1} - k_{r'}}\\
                    \end{array} \right ).
$$
\end{lemma}

\subsection{Proof of Theorem~\ref{t:invariant of resonant}}\label{ss:proof of invariant of resonant}
Let $(q_r)_{r \ge 0}$, $(b(r))_{r \ge 0}$, $Q$, $(S_k)_{k \ge 0}$, $(I_r)_{r \ge 0}$ and $( \Xi_r )_{r \ge 0}$ be as in the statement of the theorem.
We will use several times the fact that the sequence $(b(r) - r)_{r \ge 0}$ is non-decreasing.

In view of Theorem~\ref{t:measure isomorphism} and Theorem~\ref{t:reduction to Bratteli}, it is enough to prove that the space of invariant probability measures of the Bratteli-Vershik system $(X_{B_Q}, V_{B_Q})$ is affine homeomorphic to~$\varprojlim_{r} (\Delta_{I_r}, \Xi_r)$.
Suppose first that for each $r \ge 0$ be have $b(r) = r$, so that the inverse limit $\varprojlim_{r} (\Delta_{I_r}, \Xi_r)$ is reduced to a point.
Then for each $r \ge 0$ we have $Q(q_r + 1) = q_r$, and Lemma~\ref{l:unique ergodicity} and Theorem~\ref{t:reduction to Bratteli} imply that $(X_{B_Q}, V_{B_Q})$ is uniquely ergodic.
So the theorem is verified in this case.
Therefore we can assume that there is an integer $r_0 \ge 1$ such that $b(r_0) \ge r_0 + 1$.
A direct computation shows that for each $k \ge q_{b(r_0)} + 1$ we have $Q(k) \le k - 2$.
Let $\widetilde{Q} : \NN \to \NN$ be defined by,
$$ \widetilde{Q}(k) = \begin{cases}
0 & \text{if $k = 0, \ldots, q_{b(r_0)}$}, \\
Q(k) & \text{if $k \ge q_{b(r_0)} + 1$}.
\end{cases} $$
It is a resonant kneading map such that for each $k \ge 2$ we have $\widetilde{Q}(k) \le k - 2$.
In view of Lemma~\ref{l:basic Bratteli-Vershik} we may assume, replacing~$Q$ by~$\widetilde{Q}$ if necessary, that for each $k \ge 2$ we have $Q(k) \le k - 2$.
Then the resonant kneading map~$Q$ satisfies the hypothesis considered in the previous subsection.
We will keep the notation introduced there.
Note in particular that for each $r \ge 1$ we have $k_r = q_{b(r)} + 1$.

For each $r\geq 0$ denote by~$A_r$ the stochastic matrix in~$\M_{I_r, I_r}$ defined by:
\begin{itemize}
\item $A_r( \cdot ,0) = \frac{S_{q_r}}{S_{q_{r+1}}} \vec{e}_0 + (q_{r+1} - q_r) \frac{S_{Q(q_{r+1})}}{S_{q_{r+1}}} \vec{e}_1$;
\item for each $s=1,\ldots, b(r)-r$, we have $A_r(\cdot,s)= \vec{e}_{s + 1}$;
\item $A_r(\cdot, b(r)-r) = \vec{e}_0$.
\end{itemize}
We have
\begin{equation}
\label{e:det1}
\det(A_r)
=
(q_{r+1} - q_r)\frac{S_{Q(q_{r+1})}}{S_{q_{r+1}}}
=
1 - \frac{S_{q_r}}{S_{q_{r+1}}}.
\end{equation}
Moreover, denote by $\Theta_{r}:\RR^{I_{r + 1}}\to \RR^{I_{r}}$ the linear map defined by
$$
\Theta_{r}(x_0,\ldots, x_{b(r + 1)- (r + 1)})
=
\left( x_0,\ldots,x_{b(r) - r - 1}, \sum_{j=b(r) - r}^{b(r + 1)- (r + 1)}x_j \right).
$$

In view of Lemma~\ref{l:medidas Bratteli}, to prove Theorem~\ref{t:invariant of resonant} we just need to prove that $\varprojlim_{j} (\Delta_{V_j}, M_j)$ is affine homeomorphic to $\varprojlim_{r} (\Delta_{I_r}, \Xi_r)$.
We will show first that $\varprojlim_{j} (\Delta_{V_j}, M_j)$ is affine homeomorphic to $\varprojlim_{r} (\Delta_{I_r}, A_r\Theta_r)$, and then complete the proof with Lemma~\ref{l:model}, below, by showing that this last space is affine homeomorphic to $\varprojlim_{r} (\Delta_{I_r}, \Xi_r)$.

To show that $\varprojlim_{j} (\Delta_{V_j}, M_j)$ is affine homeomorphic to $\varprojlim_{j} (\Delta_{I_r}, A_r \Theta_r)$, for each $r \ge 0$ let $\Pi_r:\RR^{V_{q_r + 1}}\to \RR^{I_r}$ be the stochastic linear map defined by
$$ \Pi_r((x_{q_r + 1}, x_{q_r + 2}, \ldots, x_{q_{b(r)} + 1}))
=
\left( x_{q_r + 1}, \sum_{t=q_{r} + 2}^{q_{r + 1} + 1} x_t, \ldots, \sum_{t=q_{b(r)} + 2}^{q_{b(r)} + 1} x_t \right).
$$
Note that $\Pi_r$ maps $\Delta_{V_r}$ onto $\Delta_{I_r}$.
On the other hand, a direct computation shows that
\begin{equation}
\label{det0}
\Pi_rM_{q_r + 2} \cdots M_{q_{r+1} + 1}
=
A_r\Theta_r\Pi_{r+1}.
\end{equation}
Therefore the sequence of linear maps $(\Pi_r)_{r \ge 1}$ defines a continuous linear map
$$
\Pi : \varprojlim_{j}(\RR^{V_j}, M_j) \to \varprojlim_{j} (\RR^{I_r}, A_r\Theta_r),
$$
mapping $\varprojlim_{j}(\Delta_{V_j}, M_j)$ onto $\varprojlim_{j} (\Delta_{I_r}, A_r\Theta_r)$.

Given $r \ge 0$ observe that in Lemma~\ref{l:multiplicacion} we have $r' = b(r) - 1$, so this lemma implies that the rank of the matrix $M_{q_r + 2} \cdots M_{q_{b(r)} + 1}$ is equal to $b(r) - r + 1$, which is equal to the dimension of~$\RR^{I_r}$.
As~$\Pi$ is onto, this shows that~$\Pi$ is a homeomorphism, and that the inverse limits $\varprojlim_{j}(\Delta_{V_j}, M_j)$ and $\varprojlim_{j} (\Delta_{I_r}, A_r\Theta_r)$ are affine homeomorphic.

\begin{lemma}\label{l:model}
The spaces $\varprojlim_{r} (\Delta_{I_r}, A_r \Theta_r)$ and $\varprojlim_{r} (\Delta_{I_r}, \Xi_r)$ are affine homeomorphic.
\end{lemma}
To prove this lemma define, for each $r \ge 0$, the stochastic linear map $A_r' : \mathbb{R}^{I_{r}} \to \mathbb{R}^{I_r}$ by $$ A_r' (x_0, \ldots, x_{b(r) - r}) = \left( x_{b(r) - r}, x_0, \ldots, x_{b(r) - r - 1} \right), $$
and observe that $A_r'\Theta_r = \Xi_r$.
The proof of Lemma~\ref{l:model} is based on the following one.
\begin{lemma}\label{l:estimates}
For each $r'' \ge r' \ge r \ge 0$ and $\vec{v} \in \Delta_{I_{r''}}$ we have
\begin{multline*}
\left\| \ (A_{r} \Theta_{r} \cdots A_{r' - 1} \Theta_{r' - 1})(A_{r'}' \Theta_{r'} \cdots A_{r'' - 1}' \Theta_{r'' - 1})(\vec{v})
- \right. \\ \left. -
(A_{r}\Theta_{r} \cdots A_{r'' - 1}\Theta_{r'' - 1})(\vec{v}) \ \right\|_1,
\end{multline*}
\begin{multline*}
\left\| \ (A_{r}' \Theta_{r} \cdots A_{r' - 1}' \Theta_{r' - 1})(A_{r'} \Theta_{r'} \cdots A_{r'' - 1} \Theta_{r'' - 1})(\vec{v})
- \right. \\ \left. -
(A_{r}'\Theta_{r} \cdots A_{r'' - 1}'\Theta_{r'' - 1})(\vec{v}) \ \right\|_1
\end{multline*}
$$ \hfill \le 2 \sum_{s = r'}^{r'' - 1} (1 - \det(A_s)). $$
\end{lemma}
\begin{proof}
A direct computation shows that for each $r \ge 0$ and $\vec{v} \in \Delta_{I_r}$ we have $\| A_r(\vec{v}) - A_r'(\vec{v}) \|_1 \le 2(1 - \det (A_r))$.
The statement of the lemma is an easy consequence of this fact, and of Lemma~\ref{l:non expanding}. \end{proof}

\begin{proof}[Proof of Lemma~\ref{l:model}]
The hypothesis $\prod_{r \ge 0} \det(A_s) > 0$ implies that $\sum_{r \ge 0} (1 - \det(A_r)) < + \infty$.
So Lemma~\ref{l:estimates} implies that for each $\vec{v} \= (\vec{v}_r)_{r \ge 0}$ in $\varprojlim_{r} (\Delta_{I_r}, A_r \Theta_r)$ and each $r_0 \ge 0$, the sequence
$$ (A_{r_0}' \Theta_{r_0} \cdots A_{r - 1}' \Theta_{r - 1}(\vec{v}_r))_{r \ge r_0} $$
converges to a vector in~$\Delta_{I_{r_0}}$.
We will denote this vector by~$B_{r_0}(\vec{v})$.
It follows from Lemma~\ref{l:estimates} that the map
$$ B_{r_0} : \varprojlim_{r} (\Delta_{I_r}, A_r' \Theta_r) \to \Delta_{I_{r_0}} $$
so defined is continuous and affine.
So for each $r \ge r_0$ we have
$$ B_{r_0} =
(A_{r_0}' \Theta_{r_0} \cdots A_{r - 1}' \Theta_{r - 1}) B_{r}. $$
Therefore the sequence of affine maps~$(B_r)_{r \ge 0}$ induces a continuous affine map
$$ B : \varprojlim_{r} (\Delta_{I_r}, A_r \Theta_r) \to \varprojlim_{r} (\Delta_{I_r}, A_r' \Theta_r). $$

We define in an analogous way a continuous affine map
$$ B' : \varprojlim_{r} (\Delta_{I_r}, A_r' \Theta_r) \to \varprojlim_{r} (\Delta_{I_r}, A_r \Theta_r). $$
It follows from the definition that~$B$ and~$B'$ are inverses of each other.
This proves the lemma.
\end{proof}

\appendix
\section{Indifferent measures and equilibrium states}\label{a:ergodic theory}
The purpose of this appendix is to prove Corollary~\ref{c:ergodic theory interval}.
As the parameters given by (the proof of) the Main Theorem are such that the corresponding logistic map has a diverging kneading map, Corollary~\ref{c:ergodic theory interval} is a direct consequence Lemma~\ref{l:ergodic theory interval} below.
This lemma is an easy consequence of well\nobreakdash-known results in the literature.

Given a unimodal map~$f$ with critical point~$c$ that is of class~$C^3$ on~$[0, 1] \setminus \{ c \}$, we will say that the critical point of~$f$ is \textit{non-flat}, if there is a constant $\ell > 0$ and diffeomorphisms $\varphi, \psi : \RR \to \RR$ of class~$C^3$, such that $\varphi(c) = \psi(f(c)) = 0$ and such that for all $x \in [0, 1]$ near~$c$ we have $|\psi(f(x))| = |\varphi(x)|^\ell$.
If in addition the Schwarzian derivative of~$f$ on $[0, 1] \setminus \{ c \}$ is negative, that is, if we have
$$ \frac{f'''}{f'} - \frac{3}{2} \left(\frac{f''}{f'} \right)^2 < 0, $$
on $[0, 1] \setminus \{ c \}$, then we say that~$f$ is \textit{\Sunimodal}.

In Lemma~\ref{l:ergodic theory interval}, below, the implication $1 \Rightarrow 2$ can be shown as in~\cite[Theorem~2']{Bru96}.
\begin{lemma}\label{l:ergodic theory interval}
Let $f$ be a \Sunimodal{} map whose kneading map diverges.
Then for an invariant probability measure of~$f$ the following properties are equivalent.
\begin{enumerate}
\item[1.]
It is supported on the \pcs{} of~$f$.
\item[2.]
Its Lyapunov exponent is zero.
\item[3.]
It is an equilibrium state of~$f$ for the potential~$- \log |f'|$, and its Lyapunov exponent is zero.
\end{enumerate}
\end{lemma}
\begin{proof}
The implication $3 \Rightarrow 2$ being trivial, we just need to prove the implications $1 \Rightarrow 3$ and $2 \Rightarrow 1$.
To prove the implication $1 \Rightarrow 3$ we first observe that by Ruelle's inequality for each invariant probability measure~$\mu$ we have $h_\mu(f) - \int \log |f'| d \mu \le 0$, see~\cite{Led81,Rue78}.
As the critical point~$c$ of~$f$ is persistently recurrent by~\cite[Proposition~3.1]{Bru98}, the implication follows from~\cite[Lemma~3.10]{BruKel98}.

The implication $2 \Rightarrow 1$ follows from a standard argument.
We will only give a sketch of proof here.
Let~$\mu$ be an invariant probability measure whose support is not contained in the \pcs{} of~$f$.
We will show that the Lyapunov exponent of~$\mu$ is positive.
Without loss of generality we assume that~$\mu$ is ergodic, and take a generic point~$x_0 \in [0, 1]$ for~$\mu$ such that in addition
$$ \lim_{n \to + \infty} \tfrac{1}{n} \log |(f^n)'(x_0)| = \int \log |(f^n)'(x_0)| d \mu. $$
Let~$A$ be an interval in~$[0, 1]$ of positive measure such that the interval with the same center as~$A$ and twice the diameter is disjoint from the \pcs{} of~$f$.
As~$f$ does not have wandering intervals, it follows that for $n \ge 1$ the maximal length of a connected component of~$f^{-n}(A)$ goes to zero as $n \to + \infty$.
Using Koebe distortion theorem we conclude that there is $N > 0$ such that for each $n \ge N$ and each $x \in f^{-n}(A)$, we have $|(f^n)'(x)| > 2$.
By the ergodic theorem each generic point~$x_0$ of~$\mu$ visits~$A$ with frequency~$\mu(A)$.
We thus have,
$$ \int \log|f'| d \mu = \lim_{n \to + \infty} \tfrac{1}{n} \log |(f^n)'(x_0)| \ge \frac{\mu(A)}{N} \log 2 > 0. $$
\end{proof}

\bibliographystyle{alpha}
\bibliography{papers}

\end{document}